\documentclass[12pt, oneside]{article}   	
\usepackage{geometry}               		
\usepackage{graphicx}
	
\usepackage{amssymb}
\usepackage{subcaption}
\usepackage{fancyhdr}
\usepackage{amsmath}
\usepackage{mathtools}
\usepackage{changepage}
\usepackage{dsfont}
\usepackage{float}
\usepackage{bm}
\usepackage{amsthm}
\usepackage{color}
\usepackage{enumerate} 
\usepackage{amsmath}
\usepackage{authblk}
\usepackage{subcaption}
\usepackage{float}
\usepackage{braket}

\usepackage{tikz}
\usetikzlibrary{shapes}
\usetikzlibrary{topaths}
\usetikzlibrary{patterns}
\usetikzlibrary{decorations.pathmorphing}

\graphicspath{ {Figs/} }

\newtheorem{defn}{Definition}
\newtheorem{thm}{Theorem}
\newtheorem{lem}{Lemma}

\newtheorem{cor}{Corollary}

\newtheorem{question}{Question}

\DeclareMathOperator*{\argmax}{arg\!\max}

\DeclareMathOperator{\supp}{supp}

\newcommand{\R}{\mathds{R}}
\newcommand{\Z}{\mathds{Z}}
\newcommand{\T}{{\intercal}}
\newcommand{\mbc}{\bm{c}^{\T}}
\newcommand{\pol}{\{\bm{x}\in\R^n:A\bm{x}=\bm{b},\, D\bm{x}\leq \bm{d}\}}

\newcommand{\Depth}{2}
\newcommand{\Height}{2}
\newcommand{\Width}{2}

\author{Alexander Black$^*$}
\author{Jes\'us A. De Loera$^*$}
\author{Sean Kafer$^\star$}
\author{Laura Sanit\`a$^\diamond$}

\affil{$^*$University of California, Davis 
\\ \texttt{\{jadeloera,aeblack\}@ucdavis.edu}}

\affil{$^\star$University of Waterloo, Canada
\\ \texttt{skafer@uwaterloo.ca}}

\affil{$^\diamond$TU Eindhoven, Netherlands
\\ \texttt{l.sanita@tue.nl}}

\begin{document}

\title{On the Simplex method for $0/1$ polytopes}
\maketitle

\abstract{
We present new pivot rules for the Simplex method for LPs over $0/1$ polytopes.  We show that the number of non-degenerate steps 
taken using these rules is strongly polynomial and even linear in the dimension or in the number of variables. 
Our bounds on the number of steps are asymptotically optimal on several well-known combinatorial polytopes.
Our analysis is based on the geometry of $0/1$ polytopes and novel modifications to the classical 
Steepest-Edge and Shadow-Vertex pivot rules. We draw interesting connections between our 
pivot rules and other well-known algorithms in combinatorial optimization.}

\section{Introduction}\label{sec:introduction}

The Simplex method \cite{Dantzigbook} is one of the most popular algorithms for solving linear programs, but despite decades of study, it is still not known whether there exists a pivot rule choice that guarantees it will always reach an optimal solution with a polynomial number of pivots, or in other words, a polynomial number of \emph{basis exchanges}. In fact, a polynomial pivot rule is not even known for linear programs over $0/1$ polytopes. These are problems of the form $\max\{\bm c^\T \bm x: \bm x \in P\}$, where $P$ is a $0/1$ polytope, i.e., its extreme points have coordinates in $\{0,1\}$. This state of affairs remains despite the fact that the diameter of the graph of a $0/1$ polytope is bounded by the dimension of the polytope \cite{Nad89} (an obstacle for general polytopes) and $0/1$-linear programs arise frequently from combinatorial optimization problems.  This paper is an in-depth study of the behavior of the Simplex method for $0/1$-linear programs.

It is known (see \cite{Schulz-Weismantel-Ziegler}) that paths of (strongly) polynomial length on the 1-skeleton of a $0/1$ polytope can be constructed using \emph{any} augmentation oracle that yields an improving adjacent extreme point in (strongly) polynomial time. However, results of this type do not give insights on the performance of the Simplex method run over $0/1$-LPs because (a) selecting an improving adjacent extreme point is not any easier than solving the original LP itself; (b) using an oracle in this way may require modifying the original objective function, resulting in paths on the 1-skeleton that are not even monotone with respect to the original objective function (essentially, the path is not a Simplex path); (c) a path of (strongly) polynomial length on the 1-skeleton of a polytope does not immediately translate into a sequence of basis exchanges of polynomial length, because of degeneracy.

Degeneracy is indeed a very crucial challenge in the analysis of the Simplex method. Given a general LP, one can often assume that the LP is non-degenerate by, for example, applying perturbation techniques. However, $0/1$ polytopes are often highly degenerate and when working with $0/1$-LPs, perturbation changes $P$ into a polytope that is no longer $0/1$. Thus, perturbation cannot be performed without loss of generality to understand the behavior of the Simplex for the original $0/1$-LP. Degeneracy might make the Simplex method \emph{cycle} and hence not terminate \cite{vanderbei}. There are several pivot rules that can avoid cycling \cite{vanderbei,terlaky+zhang,MurtyLP,Terlaky2009}. However, none of these rules guarantee a polynomial bound on the number of degenerate basis exchanges, 
and they might even require an exponential number of pivots before moving to a different extreme point (this phenomenon is called \emph{stalling}). 
While cycling may be avoided easily, resolving stalling in polynomial time is as hard as solving the general problem of finding a polynomial pivot rule for the Simplex method (see e.g.,~\cite{Murty2009}). This observation holds even for $0/1$-LPs. A first stepping stone in resolving this problem is finding a pivot rule which at least guarantees a polynomial number of non-degenerate pivots. Hence, the following natural question arises as stated for instance in page 2 of \cite{Kortenkampetal97}:

\begin{question} \label{question:count_non_degenerate}
Is there a pivot rule for the Simplex method that guarantees a (strongly) polynomial number of non-degenerate pivots on $0/1$-LPs?
\end{question}

In this paper, we give a complete, affirmative answer to the above question for arbitrary $0/1$-LPs, even in the most general representation
\begin{equation}\label{lp}
	\max\set{\,\bm c^\T\bm x\ :\, A \bm x=\bm b,\, D \bm x\leq \bm d,\, \bm x\in\R^n\,},
\end{equation}
where $\bm c$, $\bm b$, $\bm d$, $A$, and $D$ are all assumed to be integral.  

In other words, we study the LP  $\max \set{\bm c^\T \bm x : \bm x \in P}$ where $P = \pol$ and the feasible region is a $0/1$ polytope of dimension $d$. 
Note that in some cases the dimension $d$ can be exponentially smaller than $n$, the number of variables of the representation. The integrality assumption 
for the input matrices and vectors is necessary for the analysis, but it does not affect the generality of our results as rational constraints can be scaled 
up to be made integral.  

\subsection*{Our results} 

Motivated by Question \ref{question:count_non_degenerate}, we analyze two types of procedures for finding optimal solutions to $0/1$-LPs that lead to an answer to this question. 
We carefully distinguish between traditional \emph{pivot rules} for the Simplex method \textendash\, which take as input a basis and output an adjacent basis, usually in the language of tableaux updates \textendash\ and the purely geometric \emph{edge rules} \textendash\, which take as input an extreme point solution and output an improving neighboring extreme point solution connected by an edge. In this paper, we will only and always use the term \textit{pivot rule} to refer to a procedure for how to transition from a basis to the next.  On the other hand, we may sometimes refer to an edge rule as simply a \textit{rule} for the sake of brevity.

This paper presents three pivot rules which require only a polynomial number of non-degenerate pivots to reach an optimal solution. Often we start by analyzing edge rules and proving that they use only a polynomial number of edges to reach an optimal solution. We then later discuss how to turn the edge selection process into a traditional Simplex pivot rule. Note that every pivot rule yields an edge rule, but the converse is not obvious. The selection of an edge at an extreme point (chosen over all possible edges) does not imply a way to choose a pivot at a basis in general.

In Section \ref{sec:steep-enough} we present our first pivot rule which is a modification of the Steepest-Edge pivot rule (see \cite{ForrestGoldfarb92} and references there). The recent work~\cite{deloera2020pivot} of the second, third, and fourth authors of the present paper analyzes a rule for $0/1$-LPs which always moves along a steepest edge, where steepness is measured according to the $1$-norm instead of the usual $2$-norm. They showed that this Steepest-Edge rule reaches an optimal solution of a $0/1$-LP in a strongly-polynomial number of steps.  

Unfortunately, this result does not directly describe the behavior of the $1$-norm Steepest-Edge pivot rule for the Simplex method, as this pivot rule is not guaranteed to follow the same path when implemented in tableaux because of degeneracy.  In particular, even for $0/1$ polytopes, the number of edges containing a vertex may be exponentially large, and the Simplex method does not consider all possible edge-directions when determining its pivot direction.  When the Simplex method with the Steepest-Edge Pivot 
Rule moves along an edge-direction at the current basic feasible solution (i.e., performs a non-degenerate pivot), that edge direction is not guaranteed 
to be the steepest one over all possible edge directions.  It is only guaranteed to be the steepest direction among the subset of directions that correspond to moving to adjacent bases (See Section~\ref{sec:steep-enough} for an example). 

Instead here we introduce the \emph{True Steepest-Edge pivot rule} for the Simplex method, which \emph{does} follow a path where every edge followed is indeed steepest in 1-norm among all edges as considered in~\cite{deloera2020pivot}. We prove the following:

\begin{thm}\label{thm:polynomial_pivots1}
On any $0/1$-LP of the form (\ref{lp}), the Simplex method with a True Steepest-Edge pivot rule 
reaches an optimal solution by performing a (strongly) polynomial number of non-degenerate pivots. 
Furthermore, when it performs a non-degenerate pivot at an extreme point 
$\bm x$, the edge along which it moves is a steepest edge-direction at $\bm x$.
\end{thm}

In Section \ref{sec:slimshadow} we present two rules which are modifications of the \emph{Shadow-Vertex pivot rule} of K. H. Borgwardt \cite{borgwardtshadow}.  In what follows we drop the word ``Vertex" from the name for sake of brevity.
We show that for $0/1$ polytopes our rules generate  monotone paths of length 
at most $n$ and $d$, respectively.  
We call our first modification the \emph{Slim Shadow rule}. As with the True Steepest-Edge pivot rule, in Section \ref{sec:slim_simplex} we extend this rule to a pivot rule which follows the same path in the feasible region.  We call this implementation the \emph{Slim Shadow pivot rule}.

\begin{thm}
\label{thm:slim_shad_Simplex}
On any $0/1$-LP of the form (\ref{lp}), the Simplex method with the Slim Shadow pivot rule reaches an optimal solution by performing no more than $n$ non-degenerate pivots.
\end{thm}

Note that this bound is in terms of $n$, the number of variables of the input  representation of~(\ref{lp}). 
However, from \cite{Nad89}, we know that the diameter of $0/1$ polytopes is at most the dimension $d$ of the polytope. 
Fortunately, by modifying the Shadow pivot rule further (see Section \ref{sec:ordershadow}), we obtain the \emph{Ordered Shadow rule}, which yields the tight bound of $d$ steps. We can in fact turn this rule into an actual pivot rule that we 
call the \emph{Ordered Shadow pivot rule}.

\begin{thm} \label{thm:polynomial_pivots2} On any $0/1$-LP of the form~(\ref{lp}) whose feasible region has dimension $d$, the Ordered Shadow pivot rule reaches an optimal solution by performing no more than $d$ non-degenerate pivots.
\end{thm}

Note that our algorithmic results yield tight monotone diameter upper bounds for all $0/1$ polytopes.
A key philosophical difference between our results and many previous results in the literature is that we embrace the 
geometry of $0/1$ polytopes as they are. We do not rely on randomization or perturbation of the input polyhedron nor 
assume a special constraint format for the $0/1$-LPs we consider.

\subsection*{Prior work and context}

The question of finding an efficient pivot rule \textendash\, one which is guaranteed to perform only a polynomial number of pivots \textendash\, is a fundamental difficulty in studying the Simplex method. There are many different pivot rules for the Simplex method. The list of pivot rules is too large to cover here in detail; we refer the reader to \cite{terlaky+zhang} for a general taxonomy. 
In this paper we present variations of two families of pivot rules.

The first is the well-known \emph{Steepest-Edge pivot rule}, which has become the most popular choice in practice. That pivot rule takes inspiration from steepest descent in smooth optimization, and the choice of the basis exchange is determined by normalizing the length of the pivot direction to find a reasonable choice of steepest direction. See \cite{ForrestGoldfarb92} for further details. We emphasize again that this pivot rule traditionally normalizes using the $2$-norm, whereas here we use the $1$-norm.

The second is the \emph{Shadow pivot rule} \cite{borgwardt}.
This famous pivot rule is often presented in terms of two-dimensional projections (shadows) of polyhedra, but it is known to be equivalent to the older Gass-Saaty parametric pivot rule (see \cite{GassSaaty1955,Jonesetal2008}). The relationship between the projection of polyhedra and parametric linear programming problems plays a key role in our work. 
K. Borgwardt \cite{borgwardt} proved that the Simplex method using the Shadow pivot rule runs in polynomial time on average. The Shadow pivot rule continues to play a significant role in the probabilistic analysis of the Simplex method due to the advent of the \emph{smoothed analysis} of the Simplex method \cite{dadush+huiberts,tengspielman,vershynin}. 

In 1972, Klee \& Minty \cite{kleeminty} showed for the first time that pivot rules may exhibit exponential behavior. 
They constructed an explicit set of  examples such that Dantzig's original pivot rule requires exponentially many steps. 
The algorithm could be tricked into 
visiting all $2^d$ vertices of a deformed cube to find a path between two nodes which are only one step apart in the $1$-skeleton of the cube. Later, Zadeh found that 
bad exponential behavior may appear even in nice families such as network flow 
problems \cite{zadeh2}.  Today, all popular pivot rules for the Simplex method, including Shadow and Steepest-Edge,
are known to require an exponential number of steps to solve some concrete ``twisted'' linear programs 
(see \cite{amentaziegler,AvisFriedmann2017,Disser-Hopp19,Friedmannetal,GOLDFARB+sit1979,Goldfarb94,Hansen+Zwick2015,Murty1980,terlaky+zhang,zadeh2,zieglerextremelps}  and references therein). 
We also know that pivot rules as decision problems are hard in the sense of complexity theory \cite{adler+papadimitriou+rubinstein,disser+skutella, fearnley+savani}.

Today, only a few highly structured families of LPs have reasonable bounds of efficiency for the Simplex method. 
A notable family of LPs is the family of network flow linear  programs which have been shown to be solvable in polynomial time by Orlin \cite{Orlin}. Dadush \& H\"ahnle \cite{DH16} \textendash\, inspired by prior work of Brunsch \& R\"oglin \cite{brunsch+Roglin} and Eisenbrand \& Vempala \cite{Eisenbrand-Vempala17} \textendash\, studied the Simplex method with the Shadow-vertex pivot rule over \emph{low-curvature} polyhedra.  Intuitively, a polyhedron is low-curvature when the hyperplanes at 
the boundary meet vertices at sharp angles, i.e., their tangent cones are slim. 
Dadush and H\"anhle obtained a diameter bound of $O(\frac{d^2}\delta\ln\frac d\delta)$ for $d$-dimensional polyhedra with 
curvature parameter $\delta\in(0,1]$. They showed that, starting from some initial vertex,  an optimal vertex can be found using an expected  $O(\frac{d^3}\delta\ln\frac d\delta)$ Simplex pivots, each requiring $O(md)$ time to compute, where $m$ 
is the number of constraints. An initial feasible solution can be found using $O(\frac{md^3}\delta\ln\frac d\delta)$ pivot steps. Their analysis of the Shadow pivot rule differs in that they study the behavior of the pivot rule in relation to the geometry of the normal cones. Borgwardt referred to this perspective as the dual Shadow vertex-pivot rule in \cite{borgwardtshadow} and used the dual perspective to prove his bounds as well. In contrast, our analysis relies purely on the primal view of the Shadow pivot rule.

 The class of $0/1$ polytopes has some of the most interesting and challenging linear programs because binary models are fundamental for combinatorial optimization, and  many of the classical polytopes in this family, such as stable set polytopes, traveling salesperson polytope, etc., 
 come from NP-hard problems. 
 The behavior of the Simplex method over $0/1$-LPs is difficult to analyze \emph{in vivo} because they are highly degenerate  and most researchers end up relying on some kind of perturbation of the polytope to derive any conclusions.  A classical result of D. Naddef~\cite{Nad89}  states that the (undirected) diameter of the graph of a $d$-dimensional $0/1$ polytope is at most $d$.  Since the Simplex  method traces monotone directed paths on this graph,  the dimension $d$ is a general tight lower bound on the number of non-degenerate pivots required for  the Simplex method, but this bound on the diameter of the graph is not necessarily achieved for concrete pivot rules. 
 A few authors have studied the length of monotone paths  of $0/1$ polytopes \cite{matsui93,rispoli1992,rispoli1998,Blanchardetal2021}. We note that $0/1$ polytopes do not have bounded curvature or small subdeterminants because the coefficients of defining inequalities can be extremely large \cite{01ziegler}.

 Consider next the family of linear programs in standard  
 equality form  $\max \{\bm c^\T \bm x : A\bm x= \bm b, \bm x\geq \bm 0\}$ where $A$ is a real $n \times m$ matrix.  
 Inspired by the work of Y. Ye \cite{yembdp}, Kitahara \& Mizuno \cite{kitaharamizuno1,kitaharamizuno2}, showed 
 that the  number of different basic feasible solutions (BFSs) generated by  this version of the Simplex method 
 using Dantzig's pivot rule is bounded by $m \lceil n \frac{\gamma}{\tau} \log (n\frac{\gamma}{\tau}) \rceil$, where $\tau$ and $\gamma$ are the 
minimum and the maximum values of all the positive elements of primal BFSs and $\lceil a \rceil$ denotes the smallest integer greater  than $a$.  
Their  results showed that the number of non-degenerate pivots performed by the Simplex method with Dantzig's rule is strongly polynomial for a 
\emph{subclass} of $0/1$-LPs: namely, those that can be expressed in standard equality form with all extreme points having variables in $\set{0,1}$ 
(e.g., the Birkhoff polytope). We stress that while every $0/1$ polyhedral region can be written in standard form by adding slack variables, at an 
extreme point such variables could potentially take values not inside $\set{0,1}$. Thus, they may not satisfy the hypothesis of \cite{kitaharamizuno1,kitaharamizuno2}. As such, the strongly polynomial bound of \cite{kitaharamizuno1,kitaharamizuno2} does not extend to all $0/1$-LPs.

 \subsection{Preliminaries and Notation}\label{sec:prelim}

We consider LPs over $0/1$ polytopes in the most general description~(\ref{lp}), i.e., $P = \pol$. The only simplifying assumptions we will make on $P$ 
are that (i) $A$ has full rank, which is an assumption that can be made without loss of generality because redundant equalities are easy to identify, 
and (ii) the set of inequalities contains $\bm x \geq \bm 0$, which again can be made without loss of generality as 
the polytope is $0/1$.  Given an integer $k$, we will use the notation $[k]$ to denote the set $\set{1,\ldots,k}$. Given a vector $\bm z$, we will use 
the notation $\bm z(i)$ to denote the $i^\text{th}$ component of $\bm z$ and given a subset $X$ of $[n]$ we will use the notation $\bm z(X)$ to denote 
the restriction of $\bm z$ to the components indexed by $X$.  
Given a matrix $Z$ with $n$ columns and a subset $X$ of $[n]$, we will use the notation $Z_X$ to denote the restriction of $Z$ to the columns indexed by $X$.

We let $d$ refer to the dimension of the feasible region $P$ in contrast to $n$, the number of variables in description (\ref{lp}).  
Given an extreme point $\bm x'$ of $P$, we define the \emph{feasible cone} at $\bm x'$ to be the set of all directions $\bm z$ such that 
${\bm x'}+\varepsilon\bm z\in P$  for some $\varepsilon >0$. That is, it is the set 
$\mathcal{C}(\bm x') = \set{\bm z \in\R^n: \,A\bm z=\bm 0,\, \overline{D}\bm z\leq \bm 0}$ 
where $\overline{D}$ denotes row submatrix of $D$ given by the indices of the inequalities of $D\bm x\leq \bm d$ 
that are tight at $\bm x'$. 
The extreme rays of the feasible cone at $\bm x'$ correspond to the \emph{edge-directions} at $\bm x'$. 

Let $\bm z^1,\ldots,\bm z^t$ be generators of the extreme rays of the feasible cone at $\bm x'$.  
Then for each $i\in[t]$, there exists a unique extreme point $\bm x^{i}$ of $P$ which is adjacent to $\bm x'$ in $P$ such that 
$\bm z^{i} = \alpha(\bm x^{i} - \bm x')$ for some positive scalar $\alpha$.  In this circumstance, we say that $(\bm x^{i} - \bm x)$ is an \emph{edge-direction} at $\bm x'$.  
We therefore say that any generator $\bm z^{i}$ of an extreme ray of $\mathcal{C}(\bm x')$ \emph{corresponds to an edge-direction at $\bm x'$ in $P$}.

An edge-direction $\bm g$ at $\bm x$ in $P$ is called a steepest edge-direction if it maximizes $\frac{\bm c^\T \bm g}{||\bm g||_1}$ among all edge-directions at $\bm x$ in $P$. We stress once more that in this paper we always use the $1$-norm to define the steepest edges, rather than the more traditional $2$-norm used in \cite{ForrestGoldfarb92} and elsewhere. 

The extreme points of $P$ and the one-dimensional faces form the \emph{graph of the polytope} $P$, also called the \emph{$1$-skeleton} of $P$. The non-degenerate steps taken by the Simplex method correspond to a path on that graph. Furthermore, we let $N_{\bm{c}}(\bm{x})$ denote the set of $\bm{c}$-improving neighbors of $\bm{x}$ for a vertex $\bm{x}$ of $P$ and objective function $\bm{c}$.

\subsubsection{An algebraic review of the Simplex method}\label{sec:algebraic_simplex}
For general background on the Simplex method and its implementation, see \cite{bertsimastsitsiklis,vanderbei,Schrijver1986}.  We will give a brief review 
of it here in the typical language of tableaux manipulation, which is how the Simplex method is implemented in practice. We have to start by putting our LP into standard equality form. Note that for the $0/1$-LPs considered in this paper, it suffices to add slack variables, as all original variables already satisfy non-negativity by assumption. Given an initial LP of the form (\ref{lp}), 

we therefore add slack variables and get an LP of the form
$$\max \set{ \bm c'^\T \bm x :\, A'\bm x = \bm b',\, \bm x \ge \bm 0,\, \bm x\in\R^{n'}\,},$$
where $A' \in \Z^{m' \times n'}$, $\bm c' \in \Z^{n'}, \bm b' \in \Z^{m'}$.
We remark here that, unlike in the work of~\cite{kitaharamizuno1}, we do not require that the added slack variables always take on $0/1$ values at vertices.  Though we put the LP into equality form for the purposes of executing the Simplex method, the performance guarantee of the pivot rule only depends on the fact that the original LP was $0/1$.
  
Since we assumed earlier that the equality matrix $A$ of our original LP has full row-rank, we have that $A'$ has full row-rank $m'$. A \emph{basis} $B$ is a subset of $[n']$ (the column indices of the matrix $A'$) of size $m'$ such that the columns of $A'$ indexed by $B$ are linearly independent.  Given a basis $B$, assume that it is an ordered set and let $B(j)$ be the $j$-th element of the set. The variables with indices in $B$ are also referred to as \textit{basic variables}, and the variables with indices in $N := \{1,\dots,n'\} \setminus B$ are referred to as the \textit{nonbasic
variables}. A basis $B$ is uniquely associated with a \emph{basic solution} $\bm x'$ defined as follows:
$\bm x'(B) = A'^{-1}_{B} \bm b'$ and $\bm x'(N) = \bm 0$. If $\bm x' \geq \bm 0$, then it is called a \emph{basic feasible solution}, 
and the corresponding basis is called a \emph{feasible basis}.

The Simplex method can be described compactly as follows;
see, e.g., \cite{bertsimastsitsiklis,vanderbei,Schrijver1986} for a more detailed treatment.
\begin{itemize}
\item Start with any feasible basis $B$ be and repeat the following steps:
  \begin{enumerate}
  \item Compute the reduced costs for the nonbasic variables
    $\bar{\bm c}'(N)^\T = \bm c'(N)^\T - \bm c'(B)^\T A'^{-1}_B A'_N$.  If $\bar{\bm c}'(N) \le \bm 0$ the basis is
    optimal: the algorithm terminates. Otherwise, choose $j :
    \bar{\bm c}'(j) > 0$. The particular choice of $j$ depends on the pivot rule used. The variable associated with column $j$ is the {\em entering} variable.
  \item Compute $\bm u = A'^{-1}_B A'_j$. If $\bm u \le \bm 0$, the optimal cost is unbounded: the algorithm terminates.
  \item If some component of $\bm u$ is positive, compute the {\em ratio test}:
    \begin{equation}
      \label{eq:ratio_test}
      r^* \coloneqq \min_{i\in[m'] \,:\, \bm u(i) > 0} \left(\frac{\bm x'(B(i))}{\bm u(i)}\right).
    \end{equation}
  \item Let $\ell$ be such that $r^* = \frac{\bm x'(B(\ell))}{\bm u(\ell)}$.
      If there is a tie, the particular choice of $\ell$ depends on the pivot rule used. The variable associate with $\ell$ is the {\em leaving} variable. Form a new basis replacing
    $B(\ell)$ with $j$. This step is called a {\em pivot}. Go back to 1.
  \end{enumerate}
\end{itemize}

At a high level, the Simplex method moves between feasible bases, and requires a pivot rule, which given a basis, decides which \emph{adjacent basis} to move to (in particular, how to choose the entering variable and the leaving variable at each step).

Each of the Simplex method's basis exchanges corresponds geometrically to a \emph{pivot direction} in which to move from the current basic feasible solution $\bm x$. Formally, note that at Step 1 of the algorithm, each index $j \in N$ yields a pivot direction $\bm z^j$ defined as $\bm z'^j(B) = A'^{-1}_BA'_j$, $\bm z'^j(j) = 1$, and $\bm z'^j(i) = 0$ for all $i\in N\setminus\set{j}$. 
The Simplex method selects a pivot direction which is improving with respect to the objective function.

If moving in a pivot direction $\bm g$ from $\bm x'$ \textit{maintains} feasibility, we call this a \emph{non-degenerate pivot}. A non-degenerate pivot implies that $\bm x' + \varepsilon \bm g$ is feasible for some $\varepsilon >0$.  In this case, our new basis yields a new basic feasible solution $\bm x''$, where $\bm x'$ and $\bm x''$ are \emph{adjacent} in the feasible region \textendash\, that is, contained in a common edge.  If instead, moving in this direction from $\bm x'$ would \textit{violate} feasibility, we call this a degenerate pivot. Degenerate pivots imply that $\bm x' + \varepsilon \bm g$ is not feasible for any $\varepsilon >0$. In this case, though the basis changes, the new basic feasible solution still corresponds to the extreme point $\bm x'$.

\section{The True Steepest-Edge Pivot Rule} \label{sec:steep-enough}

We first justify a geometric framework through which we may define a pivot rule.
Given a feasible basis $B$ of our LP in standard equality form, let $\bm x'$ be the basic feasible solution associated to $B$, and $\bm x''$ 
be the corresponding extreme point solution of $P$ (i.e., the vector obtained from $\bm x'$ by removing the slack variables). With a slight 
abuse of terminology, we say that $\bm x''$ is the vertex of $P$ associated to $B$.  

The basic feasible solution $\bm x'$ is identified by $m' + |N|$ linearly independent constraints that are tight at $\bm x'$: namely, the constraints  
$A'\bm x = \bm b'$ plus the non-negativity constraints for the nonbasic variables.
Let $\bar N \subseteq N$ be the set of indices for the nonbasic variables that are original variables for $P$ (i.e., non slack variables). 
We can naturally associate $n$ tight constraints that identify $\bm x''$ as follows: we consider (i) the equalities $A \bm x = \bm b$, plus 
(ii) the non-negativity constraints $\bm x^j \geq 0$ for $j\in \bar N$, plus (iii) the subset of inequalities associated to each nonbasic slack variable.  
It is easy to see that these constraints are linearly independent. Let $D^B$ be the submatrix of $D$ induced by the rows of the tight inequality constraints 
(ii) and (iii). We define the \emph{basic cone} associated to $B$ as 
$$\mathcal{C}(B) = \set{ \bm z \in \R^n : \, A\bm z = \bm 0,\, D^B \bm z \leq \bm 0}.$$

One observes that given a feasible basis $B$ and its corresponding basic feasible solution $\bm x'$, the available pivot directions at $\bm x'$ project to the extreme rays of $\mathcal{C}(B)$.

Clearly, the basic cone associated to $B$ contains the feasible cone at $\bm x''$, i.e.,
$\mathcal{C}(B) \supseteq \mathcal{C}(\bm x'')$, as the system defining $\mathcal C(B)$ is a relaxation of the system defining $\mathcal C({\bm x''})$.  
A generator $\bm z'$ of an extreme ray of $\mathcal{C}(B)$ is the projection of the pivot direction given by a non-degenerate basis exchange if and only if $\bm z'$ also generates an extreme ray of $\mathcal{C}(\bm x'')$. That is, $\bm z'$ is the projection of a non-degenerate pivot direction at $\bm x'$ given by the basis $B$ if and only if $\bm z'$ also corresponds to an edge-direction at $\bm x''$ in $P$. Given this, we can partially define a pivot rule in terms of the \textit{original} LP (i.e., the LP before it was put into standard equality form) by explaining which extreme ray of $\mathcal{C}(B)$ to choose as a pivot direction.  In particular, this corresponds to the choice of the variable \textit{entering} the basis.

\begin{defn}\label{defn:true_steepest_edge}
Given a $0/1$-LP $(P)$ of the form~(\ref{lp}) with feasible region $P$, let $(P')$ be the LP obtained by putting $(P)$ into standard equality form, let $B$ be the current basis of $(P')$, and let $\bm{x}$ be the vertex of $P$ associated to that basis. Let $\bm v = \bm 1 - 2\bm x$. The True Steepest-Edge pivot rule selects a pivot direction as follows:
\begin{itemize}
\item If any generator $\bm z^j$ of an extreme ray of $\mathcal{C}(B)$ satisfies $\bm v^\T\bm z^j \leq 0$ and $\bm c^\T\bm z^j >0$, then it selects $\bm z^j$;
\item Otherwise, it selects the generator $\bm z^j$ that maximizes $\frac{\bm c^\T \bm z^j}{\bm v^\T \bm z^j}$.
\end{itemize}

\end{defn}

We now show that if the True Steepest-Edge pivot rule performs  a non-degenerate pivot, $\bm z^j$ is the \textit{steepest edge-direction} at $\bm x$ in $P$. 
A key observation for our result came from~\cite{deloera2020pivot}: if $\bm x$ is an extreme point of a $0/1$-LP, each entry of $\bm x$ is either equal to its upper bound or its lower bound, implying that the entire feasible cone $\mathcal C(\bm x)$ is contained within a single orthant of $\R^n$.

\begin{lem}\label{lem:non-degenerateperformance}
Let $\bm x$ be an extreme point solution of a $0/1$-LP of the form~(\ref{lp}) and let $\bm x' = \bm x + \alpha \bm z^j$ be the extreme point solution
obtained from $\bm x$ after moving maximally along the  direction $\bm z^j$ selected according to the True Steepest-Edge pivot rule. If $\bm x'\neq \bm x$, then $\bm z^j$ corresponds to the steepest edge-direction at $\bm x$ in $P$.
\end{lem}

\begin{proof}
Consider the vector $\bm v = \bm 1 - 2\bm x$.  That is, $\bm v$ is defined by
\begin{equation} \label{crossp-facet}
    \bm v(i) \coloneqq 
        \begin{cases}
            \,1 &\text{ if }\bm x(i) = 0,\\
            -1 &\text{ if }\bm x(i) = 1.
        \end{cases}
\end{equation}
We have that $\bm v$ is contained in the same orthant $O_x$ of $\R^n$ as the cone $\mathcal C(\bm x)$, 
where $O_x$ is equal to the set of all $\bm z\in\R^n$ satisfying the inequalities
\begin{align*}
    \bm z(i)\geq 0 &\text{ for all } i \text{ such that }\bm x(i) = 0,\\
    \bm z(i)\leq 0  &\text{ for all } i \text{ such that }\bm x(i) = 1.
\end{align*}
Then by the definition of the 1-norm, for all  vectors $\bm z$ in the orthant $O_x$, $||\bm z||_1$ is precisely equal to $\bm v^\T\bm z$.  In particular, this holds for all vectors in $\mathcal C({\bm x})$, and so all rays in $\mathcal C(\bm x)$ intersect the hyperplane $H$ defined by $\bm v^\T \bm z = 1$, i.e., $H=\{\bm z : \bm v^\T \bm z = 1\}.$

Let the $d$ extreme rays of the basic cone $\mathcal C(B)$ be generated by $\bm z^1,\ldots,\bm z^d$.  If the extreme ray generated by $\bm z^{i}$ intersects $H$, then assume without loss of generality (by possibly rescaling) that $\bm z^{i}$ is in $H$.
Note then that for any generator $\bm z^{i}$ of an extreme ray of $\mathcal{C}(B)$, if $\bm z^{i}$ happens to also correspond to an edge-direction at $\bm x$ in $P$,
we have that $\bm z^{i}\in H$.

Assume without loss of generality that the direction chosen according to the True Steepest-Edge pivot rule is $\bm z^1$.  Since $\bm x'\neq \bm x$, the direction $\bm z^1$ also corresponds to an edge-direction at $\bm x$ in $P$.  Then $\bm z^1$ is in the feasible cone, and by our earlier discussion, $\bm v^\T\bm z^1 = ||\bm z^1||_1 = 1$. Now, consider the optimization problem $\cal Q$ defined by
 \begin{align}
    &\max   \bm c^\T  \bm z \nonumber \\
    &\text{s.t.} \nonumber \\
    &\bm v^\T\bm z\leq 1, && \qquad (1) \nonumber \\
    &\bm z\in \mathcal{C}(B), & & \qquad (2) \nonumber
\end{align}
and let $P_{\cal Q}$ denote its feasible region.  Note that $P_{\cal Q}$ is a polyhedron.  

By construction, any generator $\bm z^{i}\in H$ is a vertex of $P_{\cal Q}$.  
In particular, this is true for those generators $\bm z^{i}$ which also correspond to edge-directions at $\bm x$ in $P$.

We will argue that the selected direction $\bm z^1$ is an optimal solution to the LP $\cal Q$.  Clearly, it is feasible for $\cal Q$. First, suppose for the sake of contradiction that $\cal Q$ is unbounded. Then there exists an entire ray of $\mathcal C(B)$ which is contained in $P_{\cal Q}$ on which the objective function $\bm c$ is unbounded, and therefore there exists such a ray that is an extreme ray of $\mathcal C(B)$.  Let this extreme ray be generated by $\bm z^{i}$.  

Then $\bm z^{i}$ does not correspond to an edge-direction at $\bm x$ in $P$, and therefore corresponds to a degenerate pivot at $B$. Furthermore, $\bm v^\T\bm z^{i} \leq 0$, and $\bm c^\T \bm z^{i} > 0$.  However, this implies that the pivot rule would have chosen the direction $\bm z^{i}$ and not $\bm z^1$, a contradiction. Thus, the LP $\cal Q$ is not unbounded.  This implies that all generators $\bm z^{i}$ of extreme rays of $\mathcal C(B)$ satisfying $\bm c^\T\bm z^{i} > 0$ generate extreme rays that intersect $H$.  

There exists an optimal vertex $\bm y$ of $P_{\cal Q}$, and since $\bm c^\T \bm z^1 > 0$, the optimal vertex is not $\bm 0$. Thus, $\bm y$ is precisely one of the generators 
that lies in $H$.  Let these generators be $\bm z^1,\ldots,\bm z^k$.  By the fact that $\bm z^1$ was selected, we have that $\bm z^1$ maximizes $\bm c^\T\bm z$ overall $\bm z\in\set{\bm z^1,\ldots,\bm z^k}$, and so $\bm z^1$ is an optimal solution to $\mathcal{Q}$, as desired.

We will now show that $\bm z^1$ corresponds to a steepest edge-direction at $\bm x$ in $P$.  
Let $\bm z'$ be any edge-direction at $\bm x$, and without loss of generality assume $||\bm z'||_1 =1$. Then $\bm z'$
is a feasible solution to $\cal Q$, as $\bm z' \in \mathcal{C}(\bm x) \subseteq \mathcal{C}(B)$.
Since $\bm z^1$ is an optimal solution to $\mathcal{Q}$, we have that 
$$\frac{\bm c^\T\bm z^1}{\|\bm z^1\|_1} = \bm c^\T \bm z^1 \geq \bm c^\T \bm z' = \frac{\bm c^\T \bm z'}{\|\bm z'\|_1},$$
as desired.
\end{proof}

The above result shows that our pivot rule guarantees the following:  whenever we perform a 
non-degenerate pivot, this always corresponds to moving along a steepest edge-direction at the corresponding extreme-point solution of the original LP. The following result, together with Lemma~\ref{lem:non-degenerateperformance}, provides a proof of Theorem~\ref{thm:polynomial_pivots1}:

\begin{thm}[See Theorem 1 of \cite{deloera2020pivot}]
\label{thm:deloerapreviouspaper}
Given a $0/1$-LP of the form~(\ref{lp}), the monotone path between any vertex and the optimum generated by following only steepest edge-directions has a length which is strongly-polynomial in the input size of the LP.
\end{thm}

We now provide an example showing that the standard Steepest-Edge pivot rule (using the $1$-norm) does not follow the same path as the True Steepest-Edge pivot rule.  In particular, it is possible for the Steepest-Edge pivot rule to perform a non-degenerate pivot at an extreme point solution $\bm x$ where the edge corresponding to that pivot is not a steepest edge-direction at $\bm x$.  First, we recall the definition of the Steepest-Edge pivot rule in terms of the extreme ray of $\mathcal{C}(B)$.

\begin{defn}
Given an LP $(P)$ of the form~(\ref{lp}) with feasible region $P$, let $(P')$ be the LP obtained by putting $(P)$ into standard equality form, let $B$ be the current basis of $(P')$, and let $\bm x$ be the vertex of $P$ associated to $B$.
The Steepest-Edge pivot rule selects as a pivot direction the generator of $\mathcal{C}(B)$ that maximizes $\frac{\bm c^\T\bm z}{\|\bm z\|_1}$ among all generators $\bm z$ of $\mathcal{C}(B)$.
\end{defn}

Now, consider the 3-dimensional $0/1$ polytope given by the convex hull of the points $(0,0,0)$, $(0,1,0)$, $(1,0,0)$, $(1,1,0)$, and $(0,0,1)$.  This is a $0/1$ pyramid over a square.  It can be easily checked that this polytope can be minimally described by the following inequalities:
\[
     \bm x(1) +  \bm x(3) \leq 1, \quad \bm x(2) +  \bm x(3) \leq 1, \quad \bm x(1) \geq 0, \quad \bm x(2) \geq 0,\quad \bm x(3) \geq 0.
\]
Furthermore, it is clear that at the point $\bm x' =(0,0,1)$, all but the last inequality are tight.  That is, these are the inequalities yielding the feasible cone $\mathcal{C}(\bm x')$.  The edge-directions at $\bm x'$ are given by $(0,0,-1)$, $(0, 1, -1)$, $(1, 0 ,-1)$, and $(1,1,-1)$.  Given the objective function $\bm c = (50,-1,0)$ to be maximized, it can be checked that the steepest edge-direction at $\bm x'$ is $\bm z^* = (1,0,-1)$, the edge-direction that moves to the optimal solution $(1,0,0)$.  Note that $\frac{\bm c^\T\bm z^*}{\|\bm z^*\|_1} = 25$.  The only other edge-direction at $\bm x'$ satisfying $\bm c^\T \bm z > 0$ is the direction $\bm z' = (1,1,-1)$ which moves to the point $(1,1,0)$.  Note that $\frac{\bm c^\T\bm z'}{\|\bm z'\|_1} = \frac{49}{3}$, and so $\bm z'$ is not a steepest edge-direction at $\bm x'$

Now, consider a basis $B$ associated with $\bm x'$ whose corresponding basic cone is yield by the inequalities $\bm x(1)\geq 0$, $ \bm x(1) + \bm x(3) \leq 1$, and $ \bm x(2) + \bm x(3) \leq 1$.  It can be checked that the edge-direction $\bm z^*$ does not generate an extreme ray of $\mathcal{C}(B)$, but the edge-direction $\bm z'$ does.  
The pivot directions at $B$ satisfying $\bm c^\T \bm z >0$ are the direction given by $\bm z'$ and $(0,-1,0)$. Therefore, $\bm z'$ is the pivot direction chosen by the Steepest-Edge pivot rule.  This gives a non-degenerate pivot, but as observed above, it does not correspond to a steepest edge-direction at $\bm x'$.

\subsection{Implementation of the True Steepest-Edge Pivot Rule}\label{sec:tableaux}

We now briefly describe the implementation of the True Steepest-Edge pivot rule.
Recall from Section~\ref{sec:algebraic_simplex} that in the context of the Simplex algorithm, we first put our original LP into equality form by adding slack variables.

Suppose we have some feasible basis $B$.  
For each $j\in N$, let $\bm z'^j$ be defined by $\bm z'^j(B) = A'^{-1}_BA'_j$, $\bm z'^j = 1$, and $\bm z'^j(i) = 0$ for all $i\in N\setminus\set{j}$.  
These vectors give the possible pivot directions.  Let $\bm v'$ be defined by
\begin{equation*}
    \bm v'(i) \coloneqq 
        \begin{cases}
            \,1 &\text{ if }\bm x(i) = 0 \text{ and }\bm x(i) \text{ is an original variable},\\
            -1 &\text{ if }\bm x(i) = 1 \text{ and }\bm x(i) \text{ is an original variable},\\
            0 &\text{ otherwise}.
        \end{cases}
\end{equation*}
That is, $\bm v'$ is the extension $\bm v$ (as defined earlier in Section~\ref{sec:steep-enough}) to the slack variables obtained by padding it with 0's. First, if there exists $j\in N$ such that $\bar{\bm c}'(j) > 0$ and $\bm v'^\T\bm z'^j \leq 0$, we choose $j$ to be the entering column (and we choose such a $j$ arbitrarily if more than one element of $N$ satisfies this condition).  Otherwise, we choose
\begin{equation}\label{eq:pivot_eq}
j = {\argmax}_{i \in N : \,\bar{\bm c}'(i) > 0} \,\left(\frac{\bar{\bm c}'(i)}{\bm v'^\T \bm z'^{i}}\right).
\end{equation}
Again, if more than one element of $N$ satisfies this condition we can we choose such a $j$ arbitrarily.

It may be that the choice of $\ell$ in Step 4 of our earlier description of Simplex is unique. However, we may encounter degeneracy, in which case it may not be. In that situation we use a \emph{lexicographic pivot rule} to choose the leaving variable. This is a well-known way to break ties and avoid cycling~\cite{MurtyLP,Murty2009,Terlaky2009}. 

Note that during the algorithm we carry the extra ``auxiliary cost'' vector $\bm v'^\T$. 
We can think of $\bm v'^\T$ as an additional row of the tableau that needs to be updated. The vector
$\bm v'^\T$ is zero always on the slack variables added and it has only $+1, -1$ entries for the indices of the original 
variables. After each pivot, we can easily update the entries of $\bm v'^\T$. The 
\emph{original} variables must take $0$ or $1$ values because they are vertices of a $0/1$ polyhedron. If after a pivot, an original variable goes from being $0$-valued to being $1$-valued, then we change the entry value in $\bm v'^\T$ from a $1$ to a $-1$.  The opposite switch occurs when we change one of the original variables from being $1$-valued to $0$-valued.

Finally, over the years there have been improvements on the implementations of the Simplex method. 
It is well-known that a lot of the steps can be performed faster by relying on sparsity of matrices and some numerical tricks, 
such as LU factorization,  but we refer the reader to Chapter 8 of \cite{vanderbei} for details. One special detail is that we
use the 1-norm to measure how steep the edge is. In our algorithm each iteration requires knowledge of the norms $\|A_B^{-1} A_h\|_1$. 
It is worth remarking that while for the most common 2-norm Steepest-Edge pivot rule Forrest and Goldfarb \cite{ForrestGoldfarb92} 
showed how to update these vector in a fast way, here we do not offer a speedup.

\section{Modifying the Shadow Pivot Rule} \label{sec:slimshadow}

The Shadow pivot rule is a fundamental tool in the study of the average case run-time of the Simplex method initiated by Borgwardt in \cite{borgwardtshadow}. This pivot rule gave rise to several algorithmic developments, including the arrival of \emph{smoothed complexity} \cite{tengspielman,vershynin,dadush+huiberts}. Intuitively, given a LP $\max \set{\bm c^\T \bm x : \bm x \in P}$ starting at a vertex ${\bm x}^0 \in P$, they find an auxiliary vector $\bm v$ such that $\bm{v}^{\T}\bm x$ is minimized at $\bm x^0$ that allows them to project $P$ into a polygon. The two paths of this polygon guide the possible improving paths for $P$. 

We present a novel modification of the Shadow pivot rule that instead comes from the theory of \emph{monotone path polytopes} first introduced by Billera and Sturmfels in \cite{BSFiberPoly} (see Chapter $9$ of \cite{zieglerbook} for an introduction and \cite{CellString} for more details on their structure). The vertices of the monotone path polytope are in natural correspondence with paths that the modified Shadow pivot rule may choose.  Billera and Sturmfels named these paths \emph{coherent monotone paths}. We will continue to use that 
name throughout because we will allow for situations that have not been considered in earlier treatments of Shadow pivot rule.  Namely, in contrast to the original setup of \cite{borgwardtshadow,tengspielman,vershynin,dadush+huiberts} we do not require that non-degeneracy conditions hold on the LPs we consider.

\subsection{The Shadow Rule for General Polyhedra}\label{sec:shadow_general}

For the original version of Shadow pivot rule, refer to Chapter 1 of the book of Borgwardt (\cite{borgwardtshadow}). In this section, we initially consider a general LP $\max \set{\bm c^\T \bm x : \bm x \in Q}$ where $Q$ is any polyhedron in the general form. For now $Q$ is not necessarily $0/1$.

We will prove results for general LPs in Lemma \ref{lem:submax} and Corollary \ref{cor:lengthbound}. Later we will restrict those results to $0/1$-LPs to prove the bounds on the lengths of monotone paths generated by two new edge rules for $0/1$-LPs: the \emph{Slim Shadow rule} and the \emph{Ordered Shadow rule}.  Like in the original Shadow pivot rule we follow the edges guided by a shadow. Later, in Section \ref{sec:slim_simplex}, we transform these two rules into actual pivot rules.

To start, a sequence of vertices $\Gamma = [\bm x^0, \bm x^1, \bm x^2, \dots, \bm x^k]$ is called a $\bm v$-\emph{monotone path} on $Q$ for $\bm v \in \mathbb{R}^{n}$ if $(\bm x^{i-1},\bm x^{i})$ is an edge of $Q$ and $\bm v^\T \bm x^{i-1} < \bm v^\T \bm x^{i}$ for $1 \le i \le k$. In the context of Shadow rules, we have a special type of monotone paths called \emph{coherent monotone paths}, which are constructed with both $\bm c$ and $\bm v$. Taking $\bm c$ and $\bm v$ together, we obtain a projection $\pi: \R^{n} \to \R^{2}$ given by $\pi(\bm{x}) = (\bm v^\T\bm{x}, \mbc \bm{x})$. Applying this projection to $Q$ yields a polygon $\pi(Q)$. This polygon is often called a \emph{shadow} of $Q$. 

To define coherence, we require the notion of an \emph{upper path} of $\pi(Q)$, depicted in Figure \ref{fig:upperpath}. Let $F_{0}$ and $F_{1}$ denote the $\bm{e}_1$-minimal and $\bm{e}_1$-maximal faces of $\pi(Q)$ respectively. Then let $\bm u^{0}$ and $\bm u^{1}$ be the $\bm{e}_2$-maximal vertices of $F_{0}$ and $F_{1}$, respectively. By convexity, the line segment between $\bm u^{0}$ and $\bm u^{1}$ is contained in $\pi(Q)$. Every point on the polygon lies either above or below this line segment, since the segment travels from an $\bm{e}_1$-minimum to an $\bm{e}_1$-maximum. 
The upper vertices are precisely the set of vertices that lie above this segment. Formally, let $L: \R \to \R$ be the equation of the line affinely spanned by $\bm{u}^{0}$ and $\bm{u}^{1}$. Let $\bm{u}$ be a vertex of $\pi(Q)$. Then $\bm{u}$ is a vertex of the upper path precisely when $L(\bm{e}_1^\T \bm{u}) \leq \bm{e}_2^\T \bm{u}$. The upper vertices form a path from $\bm u^{0}$ to $\bm{u}^{1}$ called the upper path.

A $\bm v$-monotone path $\Gamma$ in $Q$ is called $\bm{c}$-\emph{coherent} if $\pi$ maps $\Gamma$ to the upper path of $\pi(Q)$. Namely, all vertices of $\Gamma$ must be sent to either vertices or interior points of edges in the upper path, and every vertex of the upper path must have a vertex sent to it from $\Gamma$.  Note that this latter condition implies that the first vertex, $\bm x^0$, is the $\bm c$-maximum of the $\bm v$-minimal face of $Q$. A monotone path is called \emph{coherent} if there exists some $\bm{c} \in \mathbb{R}^{n}$ such that it is $\bm{c}$-coherent. In \cite{BSFiberPoly}, they observed that not all monotone paths on a polytope are coherent, and in \cite{MPPofCP}, the first and second authors showed that even on the octahedron some monotone paths are not coherent. We reproduce that example in Figure \ref{fig:cohpath}.

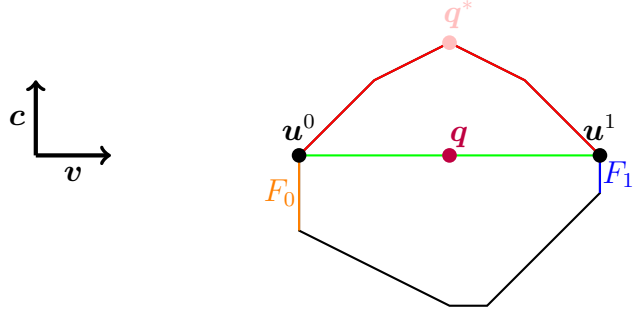
\begin{figure}
    \centering
    \[\begin{tikzpicture}[scale = .5]
    \draw[->, black, ultra thick] (-7,0) -- (-7, 2);
    \draw[->, black, ultra thick] (-7,0) -- (-5,0);
    \draw (-7.5,1) node {$\bm c$};
    \draw (-6,-.5) node {$\bm v$};
    \draw[black, thick] (0,0) -- (2,2) -- (4, 3) -- (6,2) -- (8,0) -- (8,-1) -- (6,-3) --(5,-4) -- (4,-4) -- (2,-3) -- (0,-2) -- (0,0);
    \draw[green, thick] (0,0) -- (8,0);
    \draw[red, thick] (0,0) -- (2,2) -- (4,3) -- (6,2) -- (8,0);
    \draw[orange, thick] (0,-2) -- (0,0);
    \draw (-.5,-1) node[orange] {$F_{0}$};
    \draw[blue, thick] (8,0) -- (8,-1);
    \draw (8.5, -.5) node[blue] {$F_{1}$};
    \draw (0,0) node[circle, fill, scale =.5] {};
    \draw (0,.75) node {$\bm{u}^{0}$};
    \draw (8,0) node[circle, fill, scale = .5] {};
    \draw (8,.75) node {$\bm{u}^{1}$};
    \draw[pink] (4.25, 3.75) node {$\bm{q}^{\ast}$};
    \draw[pink] (4, 3) node[circle, fill, scale = .5] {};
    \draw[purple] (4.25, 0.5) node {$\bm{q}$};
    \draw[purple] (4, 0) node[circle, fill, scale = .5] {};
\end{tikzpicture}
\]
    \caption{The red path is the upper path of the polygon. The edges, $F_{0}$ and $F_{1}$, are the $\bm{e}_1$-minimal and $\bm{e}_1$-maximal faces of the polygon respectively. The green line segment is exactly the line segment from $\bm{u}^{0}$, the $\bm{e}_2$-maximum of $F_{0}$, to $\bm{u}^{1}$, the $\bm{e}_2$-maximum of $F_{1}$. The choice of $\bm{q}^{\ast}$ and $\bm{q}$ from the proof of Lemma \ref{lem:submax} are shown as well. 
    } 
    \label{fig:upperpath} 
\end{figure}

To find a $\bm{c}$-coherent $\bm v$-monotone path in $Q$, starting at a vertex $\bm{x}^{i}$ we 
maximize the slope in the polygon $\pi(Q)$ among all $\bm v$-improving edges starting at $\bm x^{i}$. Concretely, we have 
\[\bm x^{i+1} = \argmax_{\bm u \in N_{\bm{v}}(\bm{x}^{i})}\left( \frac{\mbc(\bm u-\bm x^{i})}{\bm v^\T(\bm u-\bm x^{i})}\right).\]
 Here we make no assumption that the LP is non-degenerate. Furthermore, the choice of $\bm{x}^{i+1}$ here may not be unique when $\bm{c}$ is not generic. When it is not unique, one chooses any maximizer. From this observation, we obtain a general notion of a Shadow rule. 
\begin{defn}
Let $(Q)$ be an LP with feasible region $Q$ and objective function $\bm{c}^{\T} \bm{x}$. Then a Shadow rule constructs a path on the 1-skeleton of $Q$ from a starting vertex $\bm{x}^{0}$ to an optimal LP solution by choosing the next vertex in the path via
\[\bm x^{i+1} = \argmax_{\bm u \in N_{\bm{v}}(\bm{x}^{i})}\left( \frac{\mbc(\bm u-\bm x^{i})}{\bm v^\T(\bm u-\bm x^{i})}\right),\]
where $\bm{v}$ is an objective function such that $\bm{x}^{0}$ is the $\bm{c}$-maximum of the $\bm{v}$-minimal face of $\pi(Q)$.  
\end{defn}
Note that a Shadow rule must come with a mechanism for choosing $\bm{v}$ given an initial vertex $\bm{x}^{0}$. In the probabilistic analysis of the performance of the Simplex method using Shadow rules such as \cite{borgwardtshadow, tengspielman, DH16}, the choice of $\bm{v}$ is made randomly. In contrast, our Shadow Rules will provide a deterministic way of making this choice.

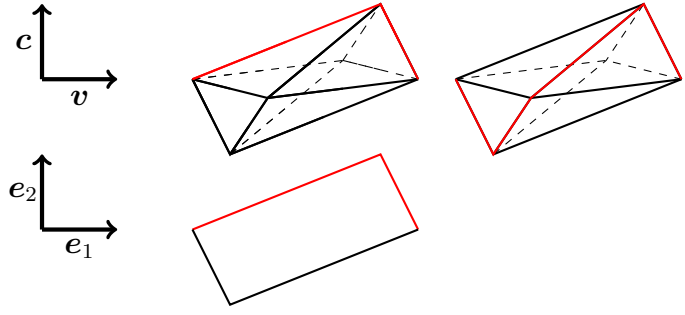
\begin{figure}
    \centering
    \begin{tikzpicture}[scale = .5]
    \draw[->, black, ultra thick] (-7,0) -- (-7, 2);
    \draw[->, black, ultra thick] (-7,0) -- (-5,0);
    \draw (-7.5,1) node {$\bm c$};
    \draw (-6,-.5) node {$\bm v$};
    \draw[->, black, ultra thick] (-7,0-4) -- (-7, 2-4);
    \draw[->, black, ultra thick] (-7,0-4) -- (-5,0-4);
    \draw (-7.5,1-4) node {$\bm{e}_{2}$};
    \draw (-6,-.5-4) node {$\bm{e}_{1}$};
    \draw[black, dashed] (-3,0) -- (-1,.25);
    \draw[black, dashed] (1,.5) -- (2,.25);
    \draw[black, dashed] (-3,0) -- (1,.5) -- (3,0);
    \draw[black, thick] (-3,0) -- (-2.5,-1);
    \draw[black, thick] (-2,-2) -- (.5,-1);
    \draw[black, thick] (-3,0) -- (-2,-2) -- (3,0);
    \draw[black, thick] (-3,0) -- (-2,-.25);
    \draw[black, thick] (-1,-.5) -- (1,-.25);
    \draw[black, thick] (-3,0) -- (-1,-.5) -- (3,0);
    \draw[black, thick] (-2,-2) -- (-1.5,-1.25);
    \draw[black, thick] (-1,-.5) -- (.5, .75);
    \draw[black, thick] (-2,-2) -- (-1,-.5) -- (2,2);
    \draw[black, dashed] (-2, -2) -- (-.5, -.75);
    \draw[black, dashed] (1.5, 1.25) -- (2,2);
    \draw[black, dashed] (1,.5) -- (1.5,1.25);
    \draw[black, dashed] (-2,-2) -- (1,.5);
    \draw[black, thick] (-3,0) -- (-2,-2) -- (-1, -.5) -- (2,2) -- (3,0);
    \draw[red, thick] (-3,0) -- (2,2) -- (3,0);
    \draw[red, thick] (-3,0) -- (-.5,1);
    \draw[red, thick] (2,2) -- (2.5,1);
    \draw[red, thick] (4-7,0-4) -- (9-7,2-4) -- (10-7,0-4);
    \draw[black, thick] (4-7,0-4) -- (5-7,-2-4) -- (10-7,0-4);
    \draw[black, thick] (-3+7,0) -- (2+7,2) -- (3+7,0);
    \draw[black, dashed] (-3+7,0) -- (1+7,.5) -- (3+7,0);
    \draw[black, thick] (-3+7,0) -- (-2+7,-2) -- (3+7,0);
    \draw[black, thick] (-3+7,0) -- (-1+7,-.5) -- (3+7,0);
    \draw[black, thick] (-2+7,-2) -- (-1+7,-.5) -- (2+7,2);
    \draw[black, dashed] (-2+7,-2) -- (1+7,.5) -- (2+7,2);
    \draw[red, thick] (-3+7,0) -- (-2+7,-2) -- (-1+7, -.5) -- (2+7,2) -- (3+7,0);
\end{tikzpicture}
    \caption{In the top of the center of the picture, a $\bm{c}$-coherent $\bm v$-monotone path is drawn in red on the octahedron $\diamond^{3}$. The corresponding shadow $\pi(\diamond^{3})$ is on the bottom of the center part of the picture where the upper path corresponding to the coherent monotone path is highlighted in red. Under the projection $\pi$, $\bm v$ and $\bm{c}$ induce the $x$ and $y$ coordinates, respectively as indicated by the arrows. On the right side of the picture is an example of an incoherent monotone path on the octahedron from \cite{MPPofCP}.}
    \label{fig:cohpath}
\end{figure}

The following lemma shows that $\bm{c}$-coherent $\bm v$-monotone paths not only lead to a maximum of $\bm v$ but also go through a maximum of $\bm c$. Note that this lemma holds for all $\bm{c}$-coherent $\bm v$-monotone paths on any polytope regardless of the choice of $\bm v$. 

\begin{lem}
\label{lem:submax}
Consider the LP $\max \{\mbc \bm{x}: \bm{x} \in Q\}$ for any polytope $Q$. Let $\bm v \in \R^{n} \setminus \{\mathbf{0}\}$, and let $\Gamma = [\bm{x}^{0}, \bm{x}^{1}, \bm x^{2}, \dots, \bm x^{k}]$ be a $\bm{c}$-coherent $\bm v$-monotone path in $Q$. Then there exists some $0 \leq i \leq k$ such that $\bm x^{i}$ maximizes $\bm c$ on $Q$ and such that the portion of the path from $\bm x^{0}$ to $\bm x^{i}$ is both $\bm{c}$-monotone and $\bm v$-monotone. 
\end{lem}

\begin{proof}
Let $\pi: Q \to \R^{2}$ be defined by $\pi(\bm{x}) = (\bm v^\T\bm{x}, \mbc\bm{x})$. Equivalently, we have $\bm{e}_{1}^{\T} \pi(\bm{x}) = \bm v^\T\bm x$ and $\bm{e}_{2}^{\T} \pi(\bm{x}) = \mbc\bm{x}$. Since $\Gamma$ is a $\bm{c}$-coherent $\bm v$-monotone path, $\pi(\Gamma)$ follows the upper path in $\pi(Q)$. Let $\bm{x}^{0}, \bm{x}^{1}, \dots, \bm{x}^{k}$ denote the vertices of $\Gamma$, and let $\bm{u}^{0}$ and $\bm{u}^{1}$ denote $\pi(\bm x^0)$ and $\pi(\bm{x}^{k})$, the first and final vertices of the upper path $\pi(\Gamma)$, respectively. As in the definition of an upper path, define $L: \R \to \R$ to be the equation of the line passing through $\bm{u}^{0}$ and $\bm{u}^{1}$. 

Let $\bm{q}^{\ast}$ be an $\bm{e}_2$-maximal vertex of $\pi(Q)$. Note that $\bm{u}^{0}$ and $\bm{u}^{1}$ are $\bm{e}_{1}$-minimal and $\bm{e}_{1}$-maximal respectively due to being the first and last vertices of the upper path. Hence, $\bm{e}_{1}^{\T} \bm{u}^{0} \leq \bm{e}_{1}^{\T}\bm{q}^{\ast} \leq \bm{e}_{1}^{\T} \bm{u}^{1}$. It follows that the point $\bm{q} = (\bm{e}_{1}^{\T} \bm{q}^{\ast}, L(\bm{e}_{1}^{\T}\bm{q}^{\ast}))$ must lie on the line segment from $\bm{u}^{0}$ to $\bm{u}^{1}$ and is therefore contained in $\pi(Q)$ by convexity. Since $\bm{q}^{\ast}$ is $\bm{e}_{2}$-maximal by assumption, $\bm{e}_{2}^{\T}\bm{q}^{\ast} \geq \bm{e}_{2}^{\T}\bm{q} = L(\bm{e}_{1}^{\T} \bm{q}^{\ast}).$ Thus, by definition, $\bm{q}^{\ast}$ must be a vertex in the upper path. 

Thus, all $\bm{e}_{2}$-maximal vertices must be part of the upper path of $\pi(Q)$.  By our assumption of coherence, there must be a vertex in the coherent monotone path $\Gamma$ that projects to a $\bm{e}_{2}$-maximal vertex in $\pi(Q)$. Then, since $\mbc \bm{x} = \bm{e}_2^\T \pi(\bm{x})$ for each $\bm{x}$ in $Q$, the $\bm{c}$-maximum of $Q$ is attained at a point $\bm{x} \in Q$ exactly when $\pi(\bm{x})$ is $\bm{e}_{2}$-maximal. Therefore, the $\bm{c}$-maximum is attained on $\Gamma$. 

By convexity, the slope from $\pi(\bm{x}^{i})$ to $\pi(\bm{x}^{i+1})$ must be at least the slope from $\pi(\bm{x}^{i+1})$ to $\pi(\bm{x}^{i+2})$ for each $0 \leq i \leq k-2$. Hence, the upper path must be strictly $\bm{e}_2$-monotone exactly until it reaches a $\bm{e}_{2}$-maximum of $\pi(Q)$. Recall again that $\bm{e}_{2}^{\T} \pi(\bm{x}^{i}) = \bm{c}^{\T} \bm{x}^{i}$ for all $0 \leq i \leq k$. Thus, $\Gamma$ must be $\bm{c}$-monotone until it reaches a $\bm{c}$-maximum on $Q$, meaning that $\Gamma$ is both $\bm v$-monotone and $\bm{c}$-monotone until it reaches a $\bm{c}$-maximum as desired.
\end{proof}

To prove Theorems~\ref{thm:slim_shad_Simplex} and~\ref{thm:polynomial_pivots2} the key idea will be to choose the auxiliary vector $\bm v$ carefully so that the corresponding shadow paths are always short. The following corollary provides a bound on the lengths of the paths.

\begin{cor}
\label{cor:lengthbound}
Let $Q \subseteq \R^{n}$ be a polytope, and let $\bm v \in \R^{n} \setminus \{\mathbf{0}\}$. Denote the set of vertices of $Q$ by $V$, and let $F$ be the face of $Q$ minimized by $\bm v$. Then for any objective function $\bm c$, the $\bm{c}$-coherent $\bm v$-monotone path from a $\bm{c}$-maximum of $F$ to a $\bm{c}$-maximum of $Q$ is of length at most $|\set{\bm v^\T\bm u: \bm u\in V}|-1$. 
\end{cor}

\begin{proof}
Since the path is strictly $\bm v$-monotone, any two vertices $\bm{x}$ and $\bm{y}$ in the path must satisfy $\bm{v}^{\T} \bm{x} \neq \bm{v}^{\T}\bm{y}$. It follows that the length of the path is at most $|\set{\bm v^\T\bm u: \bm u\in V}|-1$. Furthermore, since the path is $\bm{c}$-coherent it must reach a maximum of $\bm c$ by Lemma \ref{lem:submax}.
\end{proof}

For the original Shadow pivot rule in \cite{borgwardtshadow}, the choice of $\bm v$ is taken to be a random vector such that $\bm{v}^{\T}\bm{x}$ is minimized uniquely at the starting vertex. For our Shadow rules, we instead take advantage of the structure of $0/1$ polytopes to make this choice of $\bm v$ explicitly to guarantee that $\bm{v}^{\T} \bm{x}$ it always takes on few values, which by Corollary \ref{cor:lengthbound}, yields short paths. In essence, we try to make $\bm v$ as degenerate as possible in place of Borgwardt's generic choice. 

\subsection{The Slim Shadow Rule}\label{sec:slim_edge}

We now return to the case of $0/1$-LPs of the form $\max \set{\bm c^\T \bm x : \bm x \in P}$ where $P = \pol$, and the feasible region is a $0/1$ polytope of dimension $d$. The Slim Shadow rule is given by the following Shadow rule:

\begin{defn}\label{def:slim_edge}
Given a $0/1$-LP of the form~(\ref{lp}) with feasible region $P$, let $\bm x^0$ be any initial vertex of $P$. Let $\bm v = \bm 1 - 2\bm x^{0}$.  At a vertex $\bm x^{i}$ of $P$, the Slim Shadow rule moves to a neighbor
\[
\bm x^{i+1} = \argmax_{\bm u \in N_{\bm{v}}(\bm{x}^{i})}\left( \frac{\mbc(\bm u-\bm x^{i})}{\bm v^\T(\bm u- \bm x^{i})}\right).
\]
\end{defn}

Recall that, by Lemma \ref{lem:submax}, the maximum in this definition is always attained at a neighbor $\bm u$ satisfying $\bm c^\T \bm u > \bm c^\T \bm x^{i}$ whenever $\bm{x}^{i}$ is not $\bm{c}$-maximal.

Note that, although the Slim Shadow rule defines $\bm v$ \textit{similarly} to the way it is defined for the True Steepest-Edge pivot rule, they are not precisely the same. While for the True Steepest-Edge pivot rule we change $\bm{v}$ at each new extreme point, for the Slim Shadow rule, $\bm v$ never changes.  However, they \textit{are }defined identically at the \textit{initial} extreme point solution, and as we will see in Section~\ref{sec:slim_simplex} when we extend this rule to a pivot rule, the vector $\bm v$ also plays a similar role to the one it plays in the analysis of the True Steepest-Edge rule.

As discussed earlier, we chose this $\bm v$ because $\bm{v}^{\T}\bm{x}$ takes on very few distinct values in $0/1$-LPs.  For example, 
consider the case where $P$ is the $0/1$ cube $[0,1]^{n}$. Then for $\bm v = \bm{1}$, $\bm v^\T\bm x$ takes 
on precisely $n+1$ values at vertices of $P$, given by the possible numbers of nonzero coordinates in each vertex of the cube. 
For the Slim Shadow rule on the cube, if we choose $\bm{0}$ as our starting point, we have $\bm v = \bm{1} - 2(\bm{0}) = \bm{1}$. 
Thus, Corollary $\ref{cor:lengthbound}$ tells us that the length of a monotone path chosen by the Slim Shadow rule on the cube 
starting at the point $\bm{0}$ is at most $n$. We generalize this bound to all $0/1$ polytopes. 

\begin{thm}
\label{thm:slimshadbound}
On any $0/1$-LP of the form (\ref{lp}), where $n$ is the number of variables in the description, 
the Slim Shadow rule reaches an optimal solution by performing at most $n$ steps.
\end{thm}

\begin{proof}[Proof of Theorem \ref{thm:slimshadbound}]
Let the LP be $\max \set{\bm c^\T \bm x : \bm x \in P}$ where $P = \pol$ is a $0/1$ polytope. Let $\bm{x}^{0}$ be an initial extreme point solution. Let $S = \{ s \in [n]: \bm{x}^{0}(s) = 1\}$. Note that in the cube $[0,1]^{n}$, $\bm{x}^{0}$ is the unique minimizer of the linear function $\bm v^\T \bm x$ where $\bm v = \bm 1 - 2\bm x^0$. Hence, $\bm x^{0}$ is the unique $\bm v$-minimizer on $P$, since all vertices of $P$ are vertices of the cube.

By Corollary \ref{cor:lengthbound}, the $\bm{c}$-coherent $\bm v$-monotone path reaches the maximum of $\bm c$ from $\bm x^0$ in at most $|\{\bm{v}^{\T} \bm{x}: \bm{x} \in V\}| - 1$ steps. Let $\bar{S} = [n]\setminus S$.  
Then $-|S| \leq \bm v^\T \bm x \leq |\bar{S}|$ for all $\bm x \in \{0,1\}^{n}$, so $|\{\bm{v}^{\T} \bm{x}: \bm{x} \in V\}| \leq |S| + |\bar{S}|+1 = n + 1$. Therefore, the length of the path is at most $n$. 
\end{proof}

For a few special cases, we may tighten the bounds on the lengths of paths found by the Slim Shadow rule.
\begin{lem}
\label{lem:sparsity}
On any $0/1$-LP of the form (\ref{lp}) in which the number of nonzero coordinates among all vertices is a constant $M$, the Slim Shadow rule takes at most $M$ steps. 
\end{lem}

\begin{proof}
Let the LP be $\max \set{\bm c^\T \bm x : \bm x \in P}$ where $P = \pol$ is a $0/1$ polytope. Let $\bm{x}^{0}$ be an initial vertex, and let $S = \{s \in [n]: \bm{x}^{0}(s) = 1\}$. Then by assumption, $|S| = M$, and the linear function $-(\bm{x}^{0})^{\T}\bm x$ takes on at most $M + 1$ distinct values on $P$ given by the different possible sizes of subsets of $S$. Furthermore, by assumption, $\bm{1}^{\T}\bm x$ always yields the same value when applied to any vertex on $P$. Hence, $(\bm{1} - 2\bm{x}^{0})^\T \bm x$ takes on at most $M+1$ distinct values on vertices of $P$. Thus, by Corollary \ref{cor:lengthbound}, the length of the path used by the Slim Shadow rule starting at $\bm{x}^{0}$ is at most $M$.
\end{proof}

Note that this bound is tight when the number of nonzero coordinates $k$ is less than $n/2$, since the monotone diameter of the hyper-simplex $\Delta(n,k)$ in that case is easily verified to be $k$. For $0/1$-LPs containing $\bm{0}$ as an extreme point solution, we may improve this bound in an analogous manner. 
\begin{lem}
\label{lem:maxnumnonzero}
On any $0/1$-LP of the form (\ref{lp}) in which $\bm{0}$ is a vertex of the feasible region $P$ and in which each vertex has at most $M$ nonzero entries, the Slim Shadow rule starting at $\bm{0}$ takes at most $M$ steps.
\end{lem}

\begin{proof}
Let the LP be $\max \set{\bm c^\T \bm x : \bm x \in P}$ where $P = \pol$ is a $0/1$ polytope. Since we are starting at $\bm{0}$, we have that $\bm v = \bm{1}$. By assumption, $\bm{1}^{\T}\bm x$ takes on at most $M+1$ distinct values on vertices of $P$. Hence, the Slim Shadow rule will take at most $M$ steps.
\end{proof}

\subsection{The Ordered Shadow Rule} \label{sec:ordershadow}

The bound given for the Slim Shadow rule is given in terms of the number of variables $n$. However, 
from \cite{Nad89}, we know that the diameter of a 0/1 polytope is at most its dimension $d$. To attain 
a bound of at most $d$ steps, we introduce the \emph{Ordered Shadow rule}.

\begin{defn}
\label{def:ordshad}
Given a $0/1$-LP of the form ~(\ref{lp}) with feasible region $P$, let $\bm{x}^{0}$ be an initial extreme point on $P$. Let $c^{\ast} = ||\bm{c}||_{1} + 2$.  Define $\bm{v} \in \mathbb{R}^{n}$ by $\bm{v}(k) = (-1)^{\bm{x}^{0}(k)}(c^{\ast})^{k}$. Then the \emph{Ordered Shadow rule} is given by the Shadow rule (as described in Section \ref{sec:shadow_general}) with auxiliary vector $\bm{v}$. Explicitly, we have
\[\bm{x}^{i+1} = \argmax_{\bm{u} \in N_{\bm{v}}(\bm{x}^{i})} \frac{\bm{c}^{\T}(\bm{u} - \bm{x}^{i})}{\bm{v}^{\T}(\bm{u} - \bm{x}^{i})}.\]
\end{defn}

Like the Slim Shadow rule, the Ordered Shadow rule follows a shadow path found by choosing an auxiliary vector carefully.  Note also that Definition \ref{def:ordshad} implicitly assumes an ordering on the coordinates. Namely, we chose $|\bm{v}(k)| = (c^{\ast})^{k}$ but could also choose $|\bm{v}(k)| = (c^{\ast})^{\sigma(k)}$ for any permutation $\sigma: [n] \to [n]$, giving a different order of the variables. Different orderings of variables can yield paths of different lengths. However, for any choice of $\sigma$, $\bm{v}$ is minimized uniquely at $\bm{x}^{0}$ and the length of the path will still be at most $d$ by the same proof we provide here. To prove that the path is always short, we may no longer directly apply Corollary \ref{cor:lengthbound}. Instead, we show that the path followed satisfies another equivalent characterization for which the length of the path is easier to analyze. We first bound the length of a path satisfying this alternative characterization.

\begin{lem}
\label{lem:altchar}
Consider a $0/1$-LP of the form~(\ref{lp}) with $d$-dimensional feasible region $P$.
Let $\bm{x}^{0}$ an initial vertex in $P$. We build a monotone path $\Gamma$ on $P$ as follows:
Define $f: \R^{n} \to \{0, 1, \dots, n\}$ by $f(\bm{u}) = \max(\{k: \bm{u}(k) - \bm{x}^{0}(k) \neq 0\} \cup \{0\}).$ Given the $i$-th vertex $\bm{x}^i$ of the path, let $N_{\min}(\bm{x}^{i})$ be the set of $f$-minimal $\bm{c}$-improving neighbors of $\bm{x}^{i}$. Select the next extreme point of the path $\bm{x}^{i+1}$ as the $\bm{c}$-maximum of $N_{\min}(\bm{x}^{i})$. The length of the path $\Gamma$ constructed in this way is at most $d$.
\end{lem}

\begin{proof}
 Let the LP be $\max \set{\bm c^\T \bm x : \bm x \in P}$ where $P = \pol$ is a $0/1$ polytope, and let $\Gamma = [\bm{x}^{0}, \bm{x}^{1}, \dots]$ be the path followed by the rule described in the statement. Define $H_{k} = \{\bm{x} \in \R^{d}: \bm{x}(a) = \bm{x}^{0}(a)\text{ for all } a \geq k+1\}$ to be the plane given by fixing the last $n-k$ coordinates of a vector to agree with $\bm{x}^{0}$. Let $F_{k} = H_{k} \cap P$. Note that $[0,1]^{n} \cap H_{k}$ is a face of the $n$-cube, so since $P$ is a $0/1$ polytope, $F_{k}$ is also a face of $P$ for all $0 \leq k \leq n$. We equivalently have that $F_{k} = \{\bm{x} \in P: f(\bm{x}) \leq k\}$. 

Observe that $F_{0} = \bm{x}^{0}$, $F_{n} = P$, and $F_{k} \subseteq F_{k+1}$ for all $k \in [n]$. Consider $\bm{x}^{i}$, the $i$th vertex in the path. Suppose that, for each choice of $i \geq 0$, $\bm{x}^{i}$ and $\bm{x}^{i+1}$ are $\bm{c}$-maxima of some $F_{\alpha(i)}$ and $F_{\alpha(i+1)}$ respectively for some function $\alpha: \mathbb{N} \to [n]$. Then each vertex in the path is associated to a face. Since the path is $\bm{c}$-monotone, $F_{\alpha(i)}$ is a proper face of $F_{\alpha(i+1)}$. Therefore, because the dimension of each associated face must strictly increase, the length of the path is at most $d$. Thus, to finish the proof, it suffices to show that $\bm{x}^{i}$ and $\bm{x}^{i+1}$ are $\bm{c}$-maxima of $F_{f(\bm{x}^{i+1}) - 1}$ and $F_{f(\bm{x}^{i+1})}$ respectively. 

Suppose for the sake of contradiction that $\bm{x}^{i}$ is not a $\bm{c}$-maximum on the face $F_{f(\bm{x}^{i+1})-1}$. Then there exists a vertex $\bm{u}$ of $F_{f(\bm{x}^{i+1})-1}$ adjacent to $\bm{x}^{i}$ with $\mbc \bm{u} > \mbc \bm{x}^{i}$. Thus, we must have $f(\bm{u}) \leq f(\bm{x}^{i+1})-1 < f(\bm{x}^{i+1})$.  However, by the definition of $f$, $f(\bm{x}^{i+1})$ is minimal among all $\bm{c}$-improving neighbors of $\bm{x}^{i}$, a contradiction. Hence, $\bm{x}^{i}$ is a $\bm{c}$-maximum of $F_{f(\bm{x}^{i+1})-1}$.

 Consider the LP $\max\{\bm{c}^{\T}\bm{x}: \bm{x} \in F_{f(\bm{x}^{i+1})}\}$. It remains to show that $\bm{x}^{i+1}$ is an optimal solution to this LP. We may take $\bm{x}^{i}$ to be our initial point and argue that from that starting point, this LP may be solved by following a path of one step by the Shadow rule for auxiliary vector $\bm{e}_{f(\bm{x}^{i+1})}$ when $\bm{x}^{0}(f(\bm{x}^{i+1})) = 0$ or $-\bm{e}_{f(\bm{x}^{i+1})}$ when $\bm{x}^{0}(f(\bm{x}^{i+1})) = 1$. We then argue that $\bm{x}^{i+1}$ is a valid choice for this Shadow rule.
 
 In what remains, we shall assume that $\bm{x}^{0}(f(\bm{x}^{i+1})) = 0$, but a completely analogous argument follows when $\bm{x}^{0}(f(\bm{x}^{i+1}))=1$. Since $\bm{x}^{0}(f(\bm{x}^{i+1}))= 0$,  $F_{f(\bm{x}^{i+1})-1}$ is the $\bm{e}_{f(\bm{x}^{i+1})}$-minimal face of $F_{f(\bm{x}^{i+1})}$. We have already shown that $\bm{x}^{i}$ is a $\bm{c}$-maximum of that face. Thus, $\bm{x}^{i}$ is a valid starting point for the Shadow rule with auxiliary vector $\bm{e}_{f(\bm{x}^{i+1})
}$ for that LP.
 
 Note that $P$ is $0/1$, so $\bm{e}_{f(\bm{x}^{i+1})}^{\T} \bm{x} = \bm{x}(f(\bm{x}^{i+1}))$ takes on at most two values on vertices of $P$: $0$ or $1$. In particular, $e_{f(\bm{x}^{i+1})}^{\T}\bm{x}$ takes on at most two values on the vertices of the face $F_{f(\bm{x}^{i+1})}$. Hence, by Corollary \ref{cor:lengthbound}, a $\bm{c}$-coherent $\bm{e}_{f(\bm{x}^{i+1})}$-monotone path from a $\bm{c}$-maximum of the $\bm{e}_{f(\bm{x}^{i+1})}$-minimal face of $F_{f(\bm{x}^{i+1})}$ must be of length at most $2-1 = 1$. Hence, a $\bm{c}$-coherent $\bm{e}_{f(\bm{x}^{i+1})}$-monotone path on $F_{f(\bm{x}^{i+1})}$ starting at $\bm{x}^{i}$ must be of length at most one. 
 
 Thus, $\bm{x}^{i}$ is either equal to or adjacent to a $\bm{c}$-maximum on $F_{f(\bm{x}^{i+1})}$. Note that $\bm{x}^{i}, \bm{x}^{i+1} \in F_{f(\bm{x}^{i+1})}$, and $\mbc \bm{x}^{i+1} > \mbc \bm{x}^{i}$. Hence, $\bm{x}^{i}$ is not a $\bm{c}$-maximum on $F_{f(\bm{x}^{i+1})}$. It follows that $\bm{x}^{i}$ is adjacent to such a $\bm{c}$-maximum, and in particular, any vertex chosen by the Shadow rule on $F_{f(\bm{x}^{i+1})}$ with auxiliary vector $\bm{e}_{f(\bm{x}^{i+1})}$ must be $\bm{c}$-maximal.
 
Equivalently, a vertex $\bm{y}^{\ast}$ is $\bm{c}$-maximal on $F_{f(\bm{x}^{i+1})}$ whenever
\[\bm{y}^{\ast} \in \argmax_{\bm{u} \in N_{i+1}} \frac{\mbc(\bm{u}-\bm{x}^{i})}{\bm{e}_{f(\bm{x}^{i+1})}^{\T}(\bm{u}-\bm{x}^{i})},\]
where $N_{i+1}$ is the set of $\bm{e}_{f(\bm{x}^{i+1})}$-improving neighbors $\bm{u}$ of $\bm{x}^{i}$ in $F_{f(\bm{x}^{i+1})}$. Note that any $\bm{e}_{f(\bm{x}^{i+1})}$-improving neighbor $\bm{u}$ of $\bm{x}^{i}$ must satisfy $\bm{e}_{f(\bm{x}^{i+1})}^{\T}(\bm{u}-\bm{x}^{i}) = 1-0 =1$, since $P$ is $0/1$. Furthermore, observe that $\bm{x}^{i+1}$ must be $f$-minimal among all $\bm{c}$-improving neighbors of $\bm{x}^{i}$, since $\bm{x}^{i}$ is the $\bm{c}$-maximum of $F_{f(\bm{x}^{i+1})-1}$. Thus, because $F_{f(\bm{x}^{i+1})} = \{\bm{x}: f(\bm{x}) \leq f(\bm{x}^{i+1})\}$, all $\bm{c}$-improving neighbors of $\bm{x}^{i}$ in $F_{f(\bm{x}^{i+1})}$ must be $f$-minimal. Hence, the set of $\bm{c}$-improving neighbors of $\bm{x}^{i}$ in $F_{f(\bm{x}^{i+1})}$ is exactly $N_{\min}(\bm{x}^{i})$. It follows that
\[ \bm{x}^{i+1} \in \argmax_{\bm{u} \in N_{\min}(\bm{x}^{i})}(\mbc \bm{u}) = \argmax_{\bm{u} = N_{\min}(\bm{x}^{i})}(\mbc(\bm{u} - \bm{x}^{i})) = \argmax_{\bm{u} \in N_{i+1}} \frac{\mbc(\bm{u}-\bm{x}^{i})}{\bm{e}_{i}^{\T}(\bm{u}-\bm{x}^{i})}.\]
Therefore, $\bm{x}^{i+1}$ is a $\bm{c}$-maximum of $F_{f(\bm{x}^{i+1})}$ as desired.
\end{proof}

See Figure \ref{fig:orderedshadpath} for an example of the construction of Lemma \ref{lem:altchar}.

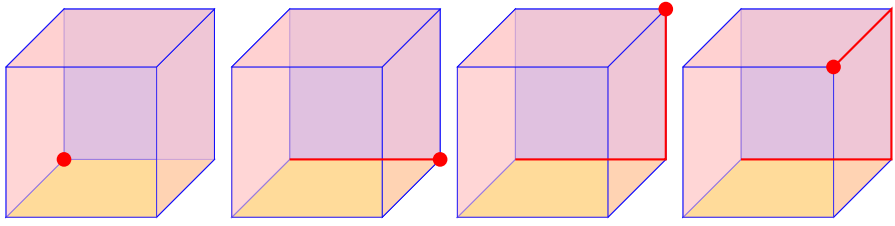
\begin{figure}[h]
    \centering
    \begin{tikzpicture}
\coordinate (O) at (0,0,0);
\coordinate (A) at (0,\Width,0);
\coordinate (B) at (0,\Width,\Height);
\coordinate (C) at (0,0,\Height);
\coordinate (D) at (\Depth,0,0);
\coordinate (E) at (\Depth,\Width,0);
\coordinate (F) at (\Depth,\Width,\Height);
\coordinate (G) at (\Depth,0,\Height);

\draw[blue,fill=yellow!80] (O) -- (C) -- (G) -- (D) -- cycle;
\draw[blue,fill=blue!30] (O) -- (A) -- (E) -- (D) -- cycle;
\draw[blue,fill=red!10] (O) -- (A) -- (B) -- (C) -- cycle;
\draw[blue,fill=red!20,opacity=0.8] (D) -- (E) -- (F) -- (G) -- cycle;
\draw[blue,fill=red!20,opacity=0.6] (C) -- (B) -- (F) -- (G) -- cycle;
\draw[blue,fill=red!20,opacity=0.8] (A) -- (B) -- (F) -- (E) -- cycle;
\draw (O) node[red, circle, fill, scale =.5] {};

\coordinate (O) at (0 + 3,0,0);
\coordinate (A) at (0 + 3,\Width,0);
\coordinate (B) at (0 + 3,\Width,\Height);
\coordinate (C) at (0 + 3,0,\Height);
\coordinate (D) at (\Depth + 3,0,0);
\coordinate (E) at (\Depth + 3,\Width,0);
\coordinate (F) at (\Depth + 3,\Width,\Height);
\coordinate (G) at (\Depth + 3,0,\Height);

\draw[blue,fill=yellow!80] (O) -- (C) -- (G) -- (D) -- cycle;
\draw[blue,fill=blue!30] (O) -- (A) -- (E) -- (D) -- cycle;
\draw[blue,fill=red!10] (O) -- (A) -- (B) -- (C) -- cycle;
\draw[blue,fill=red!20,opacity=0.8] (D) -- (E) -- (F) -- (G) -- cycle;
\draw[blue,fill=red!20,opacity=0.6] (C) -- (B) -- (F) -- (G) -- cycle;
\draw[blue,fill=red!20,opacity=0.8] (A) -- (B) -- (F) -- (E) -- cycle;
\draw[red, thick] (O) -- (D);
\draw (D) node[red, circle, fill, scale =.5] {};

\coordinate (O) at (0 + 3 + 3,0,0);
\coordinate (A) at (0 + 3 + 3,\Width,0);
\coordinate (B) at (0 + 3 + 3,\Width,\Height);
\coordinate (C) at (0 + 3 + 3,0,\Height);
\coordinate (D) at (\Depth + 3 + 3,0,0);
\coordinate (E) at (\Depth + 3 + 3,\Width,0);
\coordinate (F) at (\Depth + 3 + 3,\Width,\Height);
\coordinate (G) at (\Depth + 3 + 3,0,\Height);

\draw[blue,fill=yellow!80] (O) -- (C) -- (G) -- (D) -- cycle;
\draw[blue,fill=blue!30] (O) -- (A) -- (E) -- (D) -- cycle;
\draw[blue,fill=red!10] (O) -- (A) -- (B) -- (C) -- cycle;
\draw[blue,fill=red!20,opacity=0.8] (D) -- (E) -- (F) -- (G) -- cycle;
\draw[blue,fill=red!20,opacity=0.6] (C) -- (B) -- (F) -- (G) -- cycle;
\draw[blue,fill=red!20,opacity=0.8] (A) -- (B) -- (F) -- (E) -- cycle;
\draw[red, thick] (O) -- (D) -- (E);
\draw (E) node[red, circle, fill, scale =.5] {};

\coordinate (O) at (0 + 3 + 3 + 3,0,0);
\coordinate (A) at (0 + 3 + 3 + 3,\Width,0);
\coordinate (B) at (0 + 3 + 3 + 3,\Width,\Height);
\coordinate (C) at (0 + 3 + 3 + 3,0,\Height);
\coordinate (D) at (\Depth + 3 + 3 + 3,0,0);
\coordinate (E) at (\Depth + 3 + 3 + 3,\Width,0);
\coordinate (F) at (\Depth + 3 + 3 + 3,\Width,\Height);
\coordinate (G) at (\Depth + 3 + 3 + 3,0,\Height);

\draw[blue,fill=yellow!80] (O) -- (C) -- (G) -- (D) -- cycle;
\draw[blue,fill=blue!30] (O) -- (A) -- (E) -- (D) -- cycle;
\draw[blue,fill=red!10] (O) -- (A) -- (B) -- (C) -- cycle;
\draw[blue,fill=red!20,opacity=0.8] (D) -- (E) -- (F) -- (G) -- cycle;
\draw[blue,fill=red!20,opacity=0.6] (C) -- (B) -- (F) -- (G) -- cycle;
\draw[blue,fill=red!20,opacity=0.8] (A) -- (B) -- (F) -- (E) -- cycle;
\draw[red, thick] (O) -- (D) -- (E) -- (F);
\draw (F) node[red, circle, fill, scale =.5] {};
\end{tikzpicture}
    \caption{\small The construction of Lemma \ref{lem:altchar} yields the displayed path for maximizing $\mbc = (1,2,3)$ on the cube $[0,1]^{3}$ starting at $\bm{0}$. Observe that $\bm{0}$ is trivially the $\bm{c}$-maximum of the face in which all coordinates are fixed to be $0$. Then the path moves to $(1,0,0)$, the $\bm{c}$-maximum of the edge given by fixing the final two coordinates to equal $0$. The next step lands at $(1,1,0)$, the $\bm{c}$-maximum on the face in which the final coordinate fixed at $0$. Finally, the path ends at $(1,1,1)$, the $\bm{c}$-maximum on $[0,1]^{3}$.}
    \label{fig:orderedshadpath}
\end{figure}

We may use the bound on the path constructed in Lemma \ref{lem:altchar} to bound the length of the path followed by the Ordered Shadow rule. 

\begin{thm}\label{thm:ordered}
Consider a $0/1$-LP of the form~(\ref{lp}) with $d$-dimensional feasible region $P$. 
The Ordered Shadow rule follows a path $\Gamma$ as described in the statement of 
Lemma \ref{lem:altchar}. Hence, the number of steps taken by the Ordered Shadow rule 
to arrive at an optimal solution is at most $d$.
\end{thm}

\begin{proof}
Let $\Gamma = [\bm{x}^{0}, \bm{x}^{1}, \dots, \bm{x}^{k}]$ be the path followed by the Ordered Shadow rule. In particular, we have
\[\bm{x}^{i+1} \in \argmax_{\bm{u} \in N_{\bm{v}}(\bm{x}^{i})} \frac{\bm{c}^{\T}(\bm{u} - \bm{x}^{i})}{\bm{v}^{\T}(\bm{u}-\bm{x}^{i})}.\]
Our goal is to show that $\bm{x}^{i+1}$ must then also be a $\bm{c}$-maximal element of $N_{\min}(\bm{x}^{i})$. To prove this, we will show the following three claims hold. 

\textbf{Claim 1:} Let $\bm{u}$ be a $\bm{c}$-improving neighbor of $\bm{x}^{i}$. Then $\bm{u}$ is also $\bm{v}$-improving.

\textbf{Claim 2:} Let $\bm{u}^{0}$ and $\bm{u}^{1}$ be neighbors of $\bm{x}^{i}$ that are both $\bm{c}$-improving and $\bm{v}$-improving. Suppose that $b= f(\bm{u}^{1}) > f(\bm{u}^{0}) = a$. Then we have
\[\frac{\bm{c}^{\T}(\bm{u}^{0} - \bm{x}^{i})}{\bm{v}^{\T}(\bm{u}^{0}- \bm{x}^{i})} > \frac{\bm{c}^{\T}(\bm{u}^{1} - \bm{x}^{i})}{\bm{v}^{\T}(\bm{u}^{1}-\bm{x}^{i})}.\]

\textbf{Claim 3}: Let $\bm{u}^{0}$ and $\bm{u}^{1}$ be neighbors of $\bm{x}^{i}$ that are both $\bm{c}$-improving and $\bm{v}$-improving. Suppose that $b = f(\bm{u}^{1}) = f(\bm{u}^{0})$ and $\bm{c}^{\T}(\bm{u}^{0}) > \bm{c}^{\T}(\bm{u}^{1})$. Then we have
\[\frac{\bm{c}^{\T}(\bm{u}^{0} - \bm{x}^{i})}{\bm{v}^{\T}(\bm{u}^{0}- \bm{x}^{i})} > \frac{\bm{c}^{\T}(\bm{u}^{1} - \bm{x}^{i})}{\bm{v}^{\T}(\bm{u}^{1}-\bm{x}^{i})}.\]

Suppose that Claims $1$-$3$ are true. By Claim 1, all $\bm{c}$-improving neighbors of $\bm{x}^{i}$ are $\bm{v}$-improving, so the results of Claims $2$ and $3$ hold under the weaker assumption that $\bm{u}^{0}$ and $\bm{u}^{1}$ are any $\bm{c}$-improving neighbors. Claim $2$ shows us that when $f(\bm{u}^{1}) > f(\bm{u}^{0})$, $\bm{u}^{0}$ would be chosen over the $\bm{u}^{1}$ by the Ordered Shadow rule. Thus, $\bm{x}^{i+1}$ must be $f$-minimal among all neighbors that are $\bm{c}$-improving. That is, $\bm{x}^{i+1} \in N_{\min}(\bm{x}^{i})$. Similarly, Claim $3$ shows that the Ordered Shadow rule will always choose a neighbor with larger $\bm{c}$-value among two $f$-minimal.  That is, $\bm{x}^{i+1}$ is a $\bm{c}$-maximal element of $N_{\min}(\bm{x}^{i})$, which yields the result. Therefore, to prove the theorem, it suffices to prove Claims $1$-$3$. 

Note that each claim assumes that we start at a point $\bm{x}^{i}$ in the path. By induction, we may assume that the path up to $\bm{x}^{i}$ is of the desired type with base case satisfied for $i = 0$ by hypothesis. Equivalently, by the proof of Lemma \ref{lem:altchar}, we are able to assume that $\bm{x}^{i}$ is the $\bm{c}$-maximum of the face $F_{f(\bm{x}^{i})} = \{\bm{x} : f(\bm{x}) \leq f(\bm{x}^{i})\}$. As a result, we can apply a key assumption that all $\bm{c}$-improving neighbors of $\bm{x}^{i}$ have larger $f$-value.

We may also assume without loss of generality that $\bm{x}^{0} = \bm{0}$, which may be accomplished by a change of coordinates. In that case, $\bm{v} = (c^{*}, (c^{*})^{2}, \dots, (c^{*})^{n})$, and the following final claim will simplify our arguments for the proofs of Claims $1$-$3$.

\textbf{Claim 4:} Let $\bm{u}^{0}$ and $\bm{u}^{1}$ be neighbors of $\bm{x}^{i}$ that are both $\bm{c}$-improving and $\bm{v}$-improving. Let $b = \text{max}\{f(\bm{u}^{0}), f(\bm{u}^{1})\}$, and let 
\[\kappa_{j} = \bm{c}^{\T}(\bm{u}^{0}-\bm{x}^{i})(\bm{u}^{1}(j) - \bm{x}^{i}(j))-\bm{c}^{\T}(\bm{u}^{1}-\bm{x}^{i})(\bm{u}^{0}(j)-\bm{x}^{i}(j)).\] Then we have $\frac{\bm{c}^{\T}(\bm{u}^{0} - \bm{x}^{i})}{\bm{v}^{\T}(\bm{u}^{0}- \bm{x}^{i})} > \frac{\bm{c}^{\T}(\bm{u}^{1} - \bm{x}^{i})}{\bm{v}^{\T}(\bm{u}^{1}-\bm{x}^{i})}$ if and only if 
\[\kappa_{b}(c^{*})^{b} > \sum_{j=1}^{b-1} -\kappa_{j}(c^{*})^{j}.\]

\textit{Proof of Claim 4:} Since $\bm{u}^{0}$ and $\bm{u}^{1}$ are both $\bm{c}-$improving and $\bm{v}-$improving neighbors, the numerators and denominators of each term of the first inequality are all positive. By rearranging, we arrive at the following equivalent inequality:
\[\bm{c}^{\T}(\bm{u}^{0}-\bm{x}^{i})\bm{v}^{\T}(\bm{u}^{1}-\bm{x}^{i}) > \bm{c}^{\T}(\bm{u}^{1}-\bm{x}^{i}) \bm{v}^{\T}(\bm{u}^{0}-\bm{x}^{i}).\]
Note that $\bm{v} = (c^{*}, (c^{*})^{2}, \dots, (c^{*})^{n})$, so by evaluating $\bm{v}^{\T}(\bm{u}^{0}-\bm{x}^{i})$ and $\bm{v}^{\T}(\bm{u}^{1}-\bm{x}^{i})$ we find the inequality is equivalent to the following:
\[\bm{c}^{\T}(\bm{u}^{0}-\bm{x}^{i}) \sum_{j=1}^{n} (c^{*})^{j}(\bm{u}^{1}(j) - \bm{x}^{i}(j)) > \bm{c}^{\T}(\bm{u}^{1}-\bm{x}^{i}) \sum_{j=1}^{n} (c^{*})^{j}(\bm{u}^{0}(j) - \bm{x}^{i}(j)).\]

In what remains, it is useful to view the left and right hand sides of the equations as polynomials in $c^{\ast}$. To simplify this expression further, we will make use of the definition of $f$ from Lemma \ref{lem:altchar}. Recall that, by induction, we may assume that all neighbors of $\bm{x}^{i}$ are $f-$improving. Under this assumption, we claim that $f(\bm{u}^{1})$ and $f(\bm{u}^{0})$ are exactly the indices of the highest powers of $c^{*}$ in the left and right polynomials respectively. To see this, recall that $f(\bm{u}^{i})$ is the highest index $j$ for which $\bm{u}^{i}(j) \neq \bm{x}^{0}(j) =0$. 

By our inductive hypothesis, $f(\bm{u}^{0}) > f(\bm{x}^{i})$. Equivalently, the highest index $j$ for which $\bm{u}^{0}(j) \neq 0$ is greater than the highest index $k$ for which $\bm{x}^{i}(k) \neq 0$. Thus, $\bm{u}^{0}(j) = \bm{x}^{0}(j) = 0$ for all $j > f(\bm{u}^{0})$ by our definition of $f$. Similarly $\bm{u}^{0}(f(\bm{u}^{0})) = 1$ by the definition of $f$. At the same time $\bm{x}^{i}(f(\bm{u}^{0})) = 0$, since $f(\bm{u}^{0}) > f(\bm{x}^{i})$. Thus, we have that $\bm{u}^{0}(j) - \bm{x}^{i}(j) = 0 - 0 = 0$ for $j > f(\bm{u}^{0})$, and $\bm{u}^{0}(f(\bm{u}^{0})) - \bm{x}^{i}(f(\bm{u}^{0})) = 1 - 0 = 1$. Hence, $f(\bm{u}^{0})$ is the highest nonzero coefficient of $(c^{\ast})^{j}$ for the right polynomial. The same exact argument follows for $\bm{u}^{1}$, which establishes that it suffices to understand when the following inequality holds:
\[\bm{c}^{\T}(\bm{u}^{0}-\bm{x}^{i}) \sum_{j=1}^{f(\bm{u}^{1})} (c^{*})^{j}(\bm{u}^{1}(j) - \bm{x}^{i}(j)) > \bm{c}^{\T}(\bm{u}^{1}-\bm{x}^{i}) \sum_{j=1}^{f(\bm{u}^{0})} (c^{*})^{j}(\bm{u}^{0}(j) - \bm{x}^{i}(j)).\]
Recall that $b = \max \{f(\bm{u}^{0}), f(\bm{u}^{1})\}$. Then we may reduce the expression to understanding when
\[ \bm{c}^{\T}(\bm{u}^{0}-\bm{x}^{i})\sum_{j=1}^{b} (c^{*})^{j}(\bm{u}^{1}(j) - \bm{x}^{i}(j)) > \bm{c}^{\T}(\bm{u}^{1}-\bm{x}^{i}) \sum_{j=1}^{b} (c^{*})^{j}(\bm{u}^{0}(j) - \bm{x}^{i}(j)),\]
which we will do via order of magnitude estimates. To do these estimates, rearrange the inequality so that every term is on one side. Then we find an equation of the form 
\[\sum_{j=1}^{b} \kappa_{j} (c^{\ast})^{j} > 0, \]
where we define
\[\kappa_{j} = \bm{c}^{\T}(\bm{u}^{0}-\bm{x}^{i})(\bm{u}^{1}(j) - \bm{x}^{i}(j))-\bm{c}^{\T}(\bm{u}^{1}-\bm{x}^{i})(\bm{u}^{0}(j)-\bm{x}^{i}(j)).\]
Note that $\kappa_{j}$ is exactly the difference of terms multiplied by $(c^{\ast})^{j}$ after moving to one side. By moving all the smaller degree terms back to the other side, we finally arrive at the desired inequality:
\[\kappa_{b}(c^{*})^{b} > \sum_{j=1}^{b-1} -\kappa_{j}(c^{*})^{j}.\]

\textit{Proof of Claim 1:}
Recall that, by induction, we may assume that all $\bm{c}$-improving neighbors of $\bm{x}^{i}$ are also $f$-improving. Thus, it suffices to show that $f$-improving neighbors are $\bm{v}$-improving. Suppose that $f(\bm{u}) > f(\bm{x}^{i})$ for some $\bm{c}$-improving neighbor $\bm{u}$ of $\bm{x}^{i}$. Then $\bm{u}(f(\bm{u})) = 1$ and $\bm{x}^{i}(f(\bm{u})) = 0$, so 
\[\bm{v}^{\T}(\bm{u}-\bm{x}^{i}) = (c^{\ast})^{f(\bm{u})} + \sum_{j=1}^{f(\bm{u})-1} (\bm{u}(j)-\bm{x}^{i}(j))(c^{\ast})^{j} 
    \geq (c^{\ast})^{f(\bm{u})} - \sum_{j=1}^{f(\bm{u})-1} (c^{\ast})^{j}.\]
Since $c^{\ast} > 2$, $(c^{\ast})^{f(\bm{u})} > \sum_{j=1}^{f(\bm{u})-1} (c^{\ast})^{j}$, so 
\[\bm{v}^{\T}(\bm{u}-\bm{x}^{i}) \geq (c^{\ast})^{f(\bm{u})} - \sum_{j=1}^{f(\bm{u})-1} (c^{\ast})^{j} > 0.\] Hence, $\bm{u}$ is also $\bm{v}$-improving, which completes the proof of Claim $1$.

\textit{Proof of Claim 2:}
By the result of Claim $4$, it suffices to bound $\kappa_{b}$ from below and $|\kappa_{j}|$ from above to achieve our desired inequality. We already showed that $\bm{u}^{1}(f(\bm{u}^{1})) = 1$ and $\bm{x}^{1}(f(\bm{u}^{1})) = 0$. Furthermore, $\bm{u}^{0}(f(\bm{u}^{1})) = 0$ as well by our assumption that $f(\bm{u}^{1}) > f(\bm{u}^{0})$ and by our definition of $f$. Since $b = f(\bm{u}^{1})$, we have that
\begin{align*}
    \kappa_{b} &= \bm{c}^{\T}(\bm{u}^{0}-\bm{x}^{i})(\bm{u}^{1}(b) - \bm{x}^{i}(b))-\bm{c}^{\T}(\bm{u}^{1}-\bm{x}^{i})(\bm{u}^{0}(b)-\bm{x}^{i}(b)) \\
    &= \bm{c}^{\T}(\bm{u}^{0}-\bm{x}^{i})(1-0) - \bm{c}^{\T}(\bm{u}^{1}-\bm{x}^{i})(0-0) \\
    &= \bm{c}^{\T}(\bm{u}^{0} -\bm{x}^{i}).  
\end{align*}
Since $\bm{u}^{0}$ is $\bm{c}$-improving, $\kappa_{b} = \bm{c}^{\T}(\bm{u}^{0} - \bm{x}^{i}) > 0$. Furthermore, $\bm{c}, \bm{u}^{0},$ and $\bm{x}^{i}$ are all integer vectors, so $\kappa_{b} = \bm{c}^{\T}(\bm{u}^{0} - \bm{x}^{i}) \geq 1$. Hence, $\kappa_{b}(c^{\ast})^{b} \geq (c^{\ast})^{b}$.

For the other side of the inequality, we need to bound the sizes of lower order coefficients $|\kappa_{j}|$ for $j \leq b-1$. To do this, we need to split into cases. Suppose first that $\bm{x}^{i}(j) = 0$. Then
\begin{align*}
    \kappa_{j} &= \bm{c}^{\T}(\bm{u}^{0}-\bm{x}^{i})(\bm{u}^{1}(j) - \bm{x}^{i}(j))-\bm{c}^{\T}(\bm{u}^{1}-\bm{x}^{i})(\bm{u}^{0}(j)-\bm{x}^{i}(j)) \\
    &= \bm{c}^{\T}(\bm{u}^{0}-\bm{x}^{i})\bm{u}^{1}(j) - \bm{c}^{\T}(\bm{u}^{1}-\bm{x}^{i})\bm{u}^{0}(j).  \end{align*}

Since $\bm{u}^{0}(j), \bm{u}^{1}(j) \in \{0,1\}$, $-\bm{c}^{\T}(\bm{u}^{1}-\bm{x}^{i}) \leq \kappa_{j} \leq \bm{c}^{\T}(\bm{u}^{0}-\bm{x}^{i})$. If $\bm{x}^{i}(j) = 1$, we find by similar reasoning that $-\bm{c}^{\T}(\bm{u}^{0}-\bm{x}^{i}) \leq \kappa_{j} \leq \bm{c}^{\T}(\bm{u}^{1}-\bm{x}^{i})$. Hence, we always have that $|\kappa_{j}| \leq \text{max}\{\bm{c}^{\T}(\bm{u}^{0}-\bm{x}^{i}), \bm{c}^{\T}(\bm{u}^{1}-\bm{x}^{i})\}$. Because $\bm{u}^{0}, \bm{u}^{1}, $and $\bm{x}^{i}$ are vertices of $P$, they are in $\{0,1\}^{n}$ meaning that $\bm{u}^{0}-\bm{x}^{i}, \bm{u}^{1} - \bm{x}^{i} \in \{-1,0,1\}^{n}$. It follows that
\[|\kappa_{j}| \leq \text{max}(\bm{c}^{\T}(\bm{u}^{0}-\bm{x}^{i}), \bm{c}^{\T}(\bm{u}^{1}-\bm{x}^{i})) \leq \max_{\bm{y} \in \{-1,0,1\}^{n}} \bm{c}^{\T}\bm{y} = ||\bm{c}||_{1}.\]
Thus, we have that
\[\sum_{j=1}^{b-1} -\kappa_{i} (c^{*})^{j} \leq ||\bm{c}||_{1} \sum_{j=1}^{b-1} (c^{\ast})^{j}.\]
To finish the proof of Claim $2$, it suffices to show that $||\bm{c}||_{1} \sum_{j=1}^{b-1} (c^{\ast})^{j} < (c^{\ast})^{b}$. From a typical geometric series estimate and our choice of $c^{\ast}$, 
\[||\bm{c}||_{1} \sum_{j=1}^{b-1} (c^{\ast})^{j} = ||\bm{c}||_{1}\frac{(c^{\ast})^{b} - c^{\ast}}{c^{\ast} - 1} = \frac{||\bm{c}||_{1}}{c^{\ast} - 1}((c^{\ast})^{b} - c^{\ast}) < \frac{||\bm{c}||_{1}}{||\bm{c}||_{1}}((c^{\ast})^{b} -c^{\ast}) < (c^{\ast})^{b}.\]
Thus, Claim $2$ is true.

\textit{Proof of Claim 3:}
We again use the equivalent characterization shown in Claim $4$. Note that we have the same upper bounds for $\kappa_{i}$ as from the proof of Claim $2$ for all $i \leq b-1$ meaning that
\[\sum_{j=1}^{b-1} -\kappa_{j}(c^{*})^{j} \leq ||\bm{c}||_{1} \sum_{j=1}^{b-1} (c^{\ast})^{j}.\]
However, for $\kappa_{b}$, the situation is different. Namely, we now have $\bm{u}^{0}(b) = \bm{u}^{1}(b) = 1$, while $\bm{x}^{i}(b) = 0$. Thus, 
\begin{align*}
    \kappa_{b} &= \bm{c}^{\T}(\bm{u}^{0} - \bm{x}^{i})(\bm{u}^{1}(b) - \bm{x}^{i}(b)) - \bm{c}^{\T}(\bm{u}^{1}-\bm{x}^{i})(\bm{u}^{0}(b) -\bm{x}^{i}(b))\\
    &= \bm{c}^{\T}(\bm{u}^{0}-\bm{x}^{i}) - \bm{c}^{\T}(\bm{u}^{1} - \bm{x}^{i}) \\
    &= \bm{c}^{\T}(\bm{u}^{0}-\bm{u}^{1}).
\end{align*}
By assumption, $\bm{c}^{\T}(\bm{u}^{0}) > \bm{c}^{\T}(\bm{u}^{1})$, so $\bm{c}^{\T}(\bm{u}^{0} - \bm{u}^{1}) > 0$. By the same reasoning as before, since $\bm{c}, \bm{u}^{0},$ and $\bm{u}^{1}$ are integer vectors, we must then have $\kappa_{b} = \bm{c}^{\T}(\bm{u}^{0}-\bm{u}^{1}) \geq 1$. Then we again have $\kappa_{b} (c^{\ast})^{b} \geq (c^{\ast})^{b}$, so by the same argument as in the proof of Claim $2$, 
\[\sum_{j=1}^{b-1} -\kappa_{j} (c^{\ast})^{j} < \kappa_{b}(c^{\ast})^{b}.\]
This completes the proof.
  
\end{proof}

\subsection{A Simplex Implementation of any Shadow Rule}\label{sec:slim_simplex}

Let us start with describing the precise question that we want to address in this section. We are given a general LP (not necessarily a $0/1$-LP) with feasible region $Q$, and a feasible basis $B$ for the corresponding LP in standard equality form. Let $\bm x$ be the extreme point solution of $Q$ to which $B$ is associated. Recall that a Shadow rule chooses an improving edge direction $\bm z$ in the feasible cone $\mathcal C(\bm x)$ that maximizes $\frac{\bm c^\T \bm z}{\bm v^\T \bm z}$ for a chosen vector $\bm v$. However, if we have degenerate bases, the extreme directions of the basic cone $\mathcal C(B)$ associated with the basis $B$ may not coincide with the extreme directions of the feasible cone $\mathcal C(\bm x)$. Hence, we have the following question: How should we select an improving direction in $\mathcal C(B)$ in order to guarantee that we are following the same path on the 1-skeleton of $P$ traced by the Shadow rule? 

Our first observation is that if the initial basis $B_0$ satisfies a small assumption, then it is enough to consider the direction that maximizes the slope among the ones in the basic cone $\mathcal C(B_0)$. We first prove this, and later show that it is always possible to find a basis that satisfies the needed assumption via a sequence of degenerate pivots. Formally: 

\begin{thm}\label{thm:shadow_edge_to_pivot}
Given an LP with feasible region $Q$, let  $B^0$ be an initial feasible basis for the corresponding LP in standard equality form.
Let $\bm v$ be defined according to a given Shadow rule, and
assume that $B^0$ satisfies the following: for each $\bm z \in \mathcal C(B^0)$, if $ \bm c^\T \bm z >0$ then $\bm v^\T \bm z > 0$. 
 
Then by starting from $B^0$ and by selecting at each basis any improving pivot direction $\bm z \in \mathcal C(B)$ that maximizes $\frac{\bm c^\T\bm z}{\bm v^\T\bm z}$, the Simplex method follows the same path on the 1-skeleton of $Q$ as the one followed by the Shadow rule.

\end{thm}

\begin{proof}
 Suppose that after a number of Simplex iterations, we have reached a basis $B$ which is not optimal, and assume that every basis $B'$ 
 visited until $B$ satisfies the property that for each $\bm z \in \mathcal C(B')$, if $ \bm c^\T \bm z >0$ then $\bm v^\T \bm z > 0$.
 Note that this is true by assumption at the initial basis $B^0$. We claim that $B$ satisfies the property as well.

Suppose that at the current iteration the Shadow rule would move along an edge-direction $\bm y \in \mathcal C(B)$ .   
 Note that since $B$ is not an optimal basis with respect to $\bm c$, and since the path has not yet reached an optimum with respect to $\bm c$, Lemma \ref{lem:submax} implies that it is also not optimal with respect to $\bm v$.  Then Lemma \ref{lem:submax} further implies 
 that $\bm v^\T \bm y > 0$, and we may assume without loss of generality that $\bm v^\T \bm y = 1$.
 Let $B'$ be the basis visited right before $B$, and $\bm z'$ be the pivot direction chosen at $B'$ by our procedure.
 We know that $\bm v^\T\bm z' > 0$  and $\bm c^\T \bm z' >0$.  Then we have that the basic cone $\mathcal{C}(B)$ contains the directions $-\bm z'$ and $\bm y$.  This implies that the projection $\pi(\mathcal{C}(B))$ contains an element whose pre-image satisfies $\bm v^\T\bm z <0$ and $\bm c^\T \bm z<0$ (namely, $\pi(-\bm z')$), and a direction satisfying $\bm v^\T\bm z > 0$ and $\bm c^\T \bm z > 0$ (namely, $\pi(\bm y)$).  
Recall that, under the projection $\pi$, $\bm v$ and $\bm c$ act as coordinate vectors for the first and second coordinate of space, respectively.  Then by convexity, the cone $\pi(\mathcal{C}(B))$ does not contain any element of the orthant of $\R^2$ defined by $\bm v^\T \bm z\leq 0$, $\bm c^\T \bm z \geq 0$.  Therefore, by the definition of $\pi$, $\mathcal{C}(B)$ does not contain any element satisfying both $\bm v^\T \bm z\leq 0$ and $\bm c^\T \bm z \geq 0$.  That is, for all elements $\bm z\in\mathcal{C}(B)$, if $\bm c^\T \bm z > 0$, then $\bm v^\T \bm z > 0$. This proves our claim.

Assume that our procedure selects $\bm{\tilde{z}}$ as pivot direction for our basis $B$.
Since $\bm v^\T \bm{\tilde{z}} > 0$, we can assume without loss of generality that $\bm v^\T \bm{\tilde{z}} =1$.  We will show that $\bm{\tilde{z}}$ is an optimal solution to the following LP we call $\cal Q$:

\begin{align*}
    &\max  \bm c^\T\bm z \\
    &\text{s.t.} \\
    &\bm v^\T \bm z = 1, \\
    &\bm z\in\mathcal C (B).
\end{align*}

We first observe that $\cal Q$ is not unbounded. Since $\bm v^\T \bm z > 0$ for all $\bm z\in\mathcal{C}(B)$ satisfying $\bm c^\T\bm z > 0$, we have that the set $Z = \set{\bm z\in \mathcal{C}(B) : \bm v^\T\bm z \leq 1,\, \bm c^\T\bm z > 0}$ is a bounded set.  The set of feasible solutions to $\cal Q$ with positive objective value is contained in the set $Z$, and so it is also a bounded set.  Therefore, $\cal Q$ is not unbounded.

Since the feasible region of $\cal Q$ is just the basic cone at $B$ intersected with a single hyperplane (which does not contain the unique vertex $\bm 0$ of $\mathcal C(B)$), all extreme point solutions of $\cal Q$ correspond to generators of extreme rays of $\mathcal C(B)$ satisfying $\bm v^\T \bm z=1$.

It follows that the optimal extreme point solution of $\cal Q$ is an extreme ray generator of $\mathcal{C}(B)$ that maximizes $\frac{\bm c^\T\bm z}{\bm v^\T \bm z}$.  That is, it is the chosen pivot direction, $\bm{\tilde{z}}$.

Assume now that $\bm{\tilde{z}}$ is a non-degenerate direction.
Since $\bm y$ is the edge-direction chosen by the  Shadow rule, and since $\bm{\tilde{z}}$ in this case is also an edge-direction, 
it follows from Definition \ref{def:slim_edge} that 
$$
\bm c^\T\bm y = \frac{\bm c^\T\bm y}{\bm v^\T \bm y} \geq \frac{\bm c^\T\bm{\tilde{z}}}{\bm v^\T \bm{\tilde{z}}} =  \bm c^\T\bm{\tilde{z}}.
$$
Note that this holds because the term $\left( \frac{\bm c^\T(\bm u - \bm x^{i})}{\bm v^\T(\bm u - \bm x^{i})} \right)$ in Definition \ref{def:slim_edge} is invariant under scaling.  However, since $\bm y$ is also feasible for $\cal Q$ and since $\bm{\tilde{z}}$ is optimal for $\cal Q$, we have that in fact $\bm c^\T \bm y = \bm c^\T \bm{\tilde{z}}$.  
That is, this non-degenerate pivot corresponds to an edge-direction that the Shadow pivot rule would choose.
Thus, the Simplex method with the Shadow pivot rule follows the same path on the 1-skeleton as the Shadow rule, as desired.
\end{proof}

It remains to argue how to find a basis $B^0$ that satisfies the property needed in our previous theorem.

\begin{lem}\label{lem:get_to_B0}
Given an LP with feasible region $Q$, let $B^0$ be an initial feasible basis which is not optimal for the corresponding LP in standard equality form.
Let $\bm v$ be defined according to a given Shadow rule, and
assume that $B^0$ does not satisfy the following: for each $\bm z \in \mathcal C(B^0)$, if $ \bm c^\T \bm z >0$ then $\bm v^\T \bm z > 0$.
Then, there exists a sequence of degenerate pivots that eventually yield a basis $B$ satisfying the above condition.
\end{lem}

\begin{proof}
Let $\bm x^0$ be the extreme point associated to $B^0$. Since $B^0$ is not an optimal basis, there exists a pivot direction $\bm z \in \mathcal C(B^0)$ such that $ \bm c^\T \bm z >0$, and we can select one such that $\bm v^\T \bm z \leq 0$. Since $\pi(\bm x^0)$ lies on the upper path of the shadow of $Q$, any $\bm c$-increasing, non-degenerate pivot direction $\bm z'$ at any basis corresponding to $\bm x^0$ satisfies $\bm v^\T \bm z' > 0$.  As such, at any basis corresponding to $\bm x^0$, any $\bm c$-increasing pivot direction satisfying $\bm v^\T\bm z\leq 0$ is degenerate pivot direction.

Then suppose that we perform a series of degenerate pivots by choosing pivot directions satisfying $\bm c^\T\bm z >0$ and, if possible, $\bm v^\T\bm z\leq 0$.  Suppose that we use the lexicographic rule~\cite{MurtyLP,Murty2009,Terlaky2009} to select the variable leaving the basis. The lexicographic rule ensures that we do not cycle, so we will eventually reach a basis $B$ at which it is not possible to pick a pivot direction satisfying both $\bm c^\T \bm z > 0$ and $\bm v^\T\bm z \leq 0$, as desired.
\end{proof}

We now turn to the issue of cycling. 
In \cite{KleeKlein}, Klee and Kleinschmidt provided a method to implement Shadow pivot rules in general without having to worry about degeneracy and cycling. In essence, they showed that for any sufficiently generic choice of objective function $\bm{c}^{\T}$ and any sufficiently generic choice of auxiliary vector $\bm v$, the Shadow pivot rule with an implementation they provide does not cycle. 
For implementing the Ordered Shadow and Slim Shadow pivot rules, we may choose $\bm{c}$ to be sufficiently generic by perturbing the objective function. Furthermore, for obtaining the Ordered Shadow pivot rule, we may perturb $\bm{v}$ by increasing the value of $c^{\ast}$ in Definition \ref{def:ordshad} by any small $\varepsilon > 0$. Thus, their implementation of Simplex yields the Ordered Shadow pivot rule if we allow an additional step of perturbing $\bm{v}$. However, for the Slim Shadow rule, such a perturbation would affect our argument for bounding the length. Hence, we may not apply their implementation in that case.

As explained in \cite{MurtyLP,Murty2009,Terlaky2009}, the lexicographic pivot rule provides a general technique to avoid cycling. When the entering variable is already chosen all we need is to select  the leaving variable lexicographically. In particular, this can be integrated with the Shadow pivot rule and we do not need to assume non-degeneracy. The lexicographic rule is exactly what is used in our method for the True Steepest-Edge pivot rule. Namely, one may attach a lexicographic rule for choosing the outgoing variable to prevent cycling. This method applies to our case for both pivot rules, since it makes no assumption about $\bm{c}$ or $\bm{v}$. Hence, the Ordered Shadow and Slim Shadow pivot rules can be implemented correctly using the lexicographic method.

However, we close this section by introducing another implementation that is a hybrid between Klee and Kleinschmidt's implementation and lexicographic. Specifically, we now show that cycling can also be avoided if (a) one uses a lexicographic pivot rule to select the leaving variable until the first non-degenerate pivot, and (b) one imposes the assumption that the objective function $\bm c$ is generic.  In particular, we require that given any two pivot directions $\bm z'$ and $\bm z''$ chosen by the Shadow pivot rule, $\frac{\bm c^\T\bm z'}{\bm v^\T \bm z'} \neq \frac{\bm c^\T\bm z''}{\bm v^\T\bm z'}$.  Note that this only needs to hold for pivot directions with $\bm v^\T\bm z>0$ since these are the only pivot directions chosen by the Shadow pivot rule.

If this assumption of genericity does not hold, it can be achieved by, as usual, randomly perturbing $\bm c$ by a small amount.  That is, by replacing $\bm c$ by a new objective function $\bm c'$ chosen uniformly at random from the $\varepsilon$-ball centered at $\bm c$ (for a sufficiently small choice of $\varepsilon$).  Note that there are finitely many possible pivot directions (when normalized so that they satisfy $\bm v^\T\bm z = 1$).  If $\bm c^\T \bm z' = \bm c^\T \bm z''$, then $\bm c^\T(\bm z' - \bm z'') = 0$.  As there are only finitely many possible pivot directions of the above form, there are finitely many possible vectors of the form $\bm w = (\bm z' - \bm z'')$ for distinct pivot directions $\bm z'$ and $\bm z''$.  Then these vectors give a finite collection of $(n-1)$-dimensional linear spaces, each defined by $\bm w^\T \bm x = 0$ for a choice of $\bm w$.  Since $\bm c$ is obtained as a random element of an $n$-dimensional set, we have that with probability 1, $\bm c$ does not lie in any of these linear spaces.  That is, $\bm c^\T(\bm z' - \bm z'') \neq 0$, as desired.

Under this genericity assumption, the particular choice of the new basis (when there are ties for the leaving variable) follows the lexicographic pivot rule until we make our first non-degenerate pivot. After that, the choice of the leaving variable is arbitrary. The following lemma implies that the Shadow pivot rule with a generic objective function does not cycle:

\begin{lem}\label{lem:decreasing_slope}
Given an LP of the form~(\ref{lp}) with a generic objective function, in each iteration of the Shadow pivot rule, the value of $\frac{\bm c^\T \bm z^j}{\bm v^\T \bm z^j}$ of the chosen pivot direction $\bm z^j$ is strictly less than that of the previous iteration.
\end{lem}
\begin{proof}
Consider an arbitrary iteration at a basis $B$.  Assume without loss of generality that for each generator $\bm z^{i}$ of $\mathcal{C}(B)$, if $\bm v^\T \bm z^{i} > 0$, then $\bm v^\T \bm z^{i} = 1$.  Note that under this assumption, we seek to show that the value of $\bm c^\T \bm z^j$ is strictly less than that of the previous iteration.

Recall that the pivot direction $\bm z^j$ chosen by the Shadow pivot rule always has the property that $\pi(\bm z^j)$ is an extreme ray of $\pi(\mathcal{C}(B))$, where $\pi$ is the projection defined in Section~\ref{sec:shadow_general}.  For an element $\bm z$ of $\mathcal{C}(B)$, $\frac{\bm c^\T \bm z}{\bm v^\T \bm z}$ is precisely the slope of the ray of $\pi(\mathcal{C}(B))$ generated by $\pi(\bm z)$, and by the convexity of $\pi(\mathcal{C}(B))$, the slopes of consecutive pivot directions are non-increasing.  By our assumption of genericity, the slopes of consecutive pivot directions are not equal, and so they are decreasing, as desired.
\end{proof}

Together, Lemma \ref{lem:get_to_B0} and Lemma \ref{lem:decreasing_slope} allow us to find our initial basis and ensure that we do not cycle.  This allows us to prove Theorem \ref{thm:slim_shad_Simplex} 
and Theorem \ref{thm:polynomial_pivots2}.

\begin{proof}[Proof of Theorems \ref{thm:slim_shad_Simplex} and \ref{thm:polynomial_pivots2}]
The number of non-degenerate pivots performed by the Slim (resp. Ordered) Shadow pivot rule is precisely the number of edges in the path it takes on the 1-skeleton of the feasible region.  It follows from Theorems ~\ref{thm:slimshadbound} (resp.~\ref{thm:ordered}) and Theorem \ref{thm:shadow_edge_to_pivot} that the number of non-degenerate pivots is therefore at most $n$ (resp. $d$).
\end{proof}

We take a moment to remark upon the similarities in the definitions of the True Steepest-Edge pivot rule and the Slim Shadow pivot rule.  Their initial choices of the vector $\bm v$ coincide when they are given the same initial solution, and there are only two main differences in their definitions:
\begin{enumerate}
    \item The True Steepest-Edge pivot rule updates $\bm v$ at each new extreme point solution, while the Slim Shadow rule does not, and
    \item The True Steepest-Edge pivot rule may choose a pivot direction $\bm z$ satisfying $\bm v^\T \bm z \leq 0$, while the Slim Shadow rule will not (except possibly during degenerate pivots at the initial extreme point solution).
\end{enumerate}

It is interesting that these two pivot rules bear such similarity, especially considering their unrelated origins.

\section{Conclusions, Connections, and Comparisons}
\label{sec:conclusions}

We wish to remark that there are several examples of well-known combinatorial algorithms that turn out to use exactly the same choice of improving steps as the pivot rules presented in this paper, either in general or in special cases. 
In fact, while Theorem~\ref{thm:polynomial_pivots1} only shows that the True Steepest-Edge pivot rule reaches an optimal solution within a strongly-polynomial number of non-degenerate steps, one can get a more refined bound on the number of steps for some well-known classes of polytopes, by realizing that classical algorithms for famous combinatorial optimization problems can be interpreted as moving along steepest edges on the 1-skeleton of the $0/1$ polytope given by the set of feasible solutions.

The first example is the \emph{shortest augmenting path algorithm} for the maximum matching and maximum flow problems.
Specifically, the seminal work of Edmonds and Karp in ~\cite{10.1145/321694.321699} gave the first strongly-polynomial time algorithm for the 
maximum flow problem, showing that one can augment a given flow using augmenting paths of shortest possible length (i.e., with the minimum number of edges). 
Since then, the idea of using shortest augmenting paths has been widely used in various contexts, such as for the maximum matching problem. Note that 
augmenting a given matching by switching the edges along an augmenting path corresponds to moving between adjacent extreme points of the matching polytope~\cite{CHVATAL1975138}. Therefore, computing a maximum matching using the shortest augmenting path algorithm corresponds (from a polyhedral 
perspective) to moving along steepest edge-directions on the 1-skeleton of the matching polytope.

Another example is the \emph{minimum mean cycle canceling algorithm} by Goldberg and Tarjan~\cite{10.1145/76359.76368}.  
This algorithm finds a minimum cost circulation in a directed graph by pushing flow along cycles whose ratio of cost to number of edges is minimal. 
In general, pushing flow along cycles corresponds to moving along \emph{circuit}-directions of the corresponding circulation polytope and, as 
proved in ~\cite{deloera2020pivot}, for $0/1$ polytopes a circuit-direction whose ratio of cost to its 1-norm is minimal corresponds to a steepest edge-direction 
at a given vertex. Hence, computing a minimum cost circulation using the minimum mean cycle canceling algorithm corresponds (from a polyhedral perspective) 
to moving along steepest edge-directions on the 1-skeleton of the $0/1$ circulation polytope.

Similarly, the paths followed by the modified Shadow pivot rules specialize to well known optimization algorithms. 
Consider the \emph{greedy algorithm for optimization on matroids}. Denote by $\mathcal{I}$ be the set of independent sets of a matroid 
on a ground set $E$. Recall the $0/1$ matroid polytope associated 
to $\mathcal{I}$ is $P_{\mathcal{I}} = \text{conv}\left\{ \sum_{s \in S} \bm{e}_{s}: S \in \mathcal{I} \right\}.$

Consider the linear program $\text{max}(\bm{c}^{\T} \bm x: \bm x \in P_{\mathcal{I}})$ for a matroid polytope $P_{\mathcal{I}}$ on a ground set $E$. Let $[\bm{0} = \bm x^{0} , \bm x^{1} , \dots , \bm x^{k}]$  the path followed by the Slim Shadow rule for this LP. Let $\emptyset = S_{0} \subsetneq S_{1} \subsetneq \dots \subsetneq S_{k}$ be the sequence of subsets chosen by the greedy algorithm. Then our goal is to show that $S_{i} = \supp(\bm x^{i})$ for all $0 \leq i \leq k$. At $S_{i}$, the greedy algorithm then says to add the highest weight element $j$ not in $S_{i}$ such that $S_{i} \cup \{j\}$ is still independent. Similarly, the Slim Shadow vertex pivot rule chooses $\bm{x}^{i+1}$ as follows:

\[\bm x^{i+1} = \argmax_{\bm u\in N_{\bm{1}}(\bm{x}^{i})} \frac{\mbc(\bm u - \bm x^{i}) }{ \bm{1}^{\T}(\bm u - \bm x^{i}) } = \argmax_{\bm u\in N_{\bm{1}}(\bm{x}^{i})} \mbc(\bm u - \bm x^{i}).\]
All $\bm{1}$-improving neighbors of $\bm{x}^{i}$ are given by $\bm{x}^{i} + \bm{e}_{j}$ for some $j \notin \text{supp}(\bm x^{i})$ such that $\bm{x}^{i} + \bm e_{j}$ is still a vertex of $P$. Thus, $\bm{x}^{i+1}$ is given by maximizing $\bm{c}^{\T}(\bm{u} - \bm{x}^{i}) = \bm{c}^{\T} \bm e_{j} = \bm c(j)$ over all possible choices of $j$, which yields the result. The greedy algorithm also reflects the path chosen by True Steepest-Edge. At a greedily chosen vertex, all improving neighbors again correspond to adding some $\bm{e}_{j}$. Normalizing does not change the weight, so True Steepest-Edge corresponds also to maximizing $\bm c(j)$ over all options for $j$.

Special cases of the Ordered Shadow pivot rule paths also appear in the literature. Consider the stable set polytopes of the
complements of chordal graphs. This $0/1$  polytope is the convex hull of all $0/1$ incidence vectors of the cliques $\mathcal{S}(G)$ on a
chordal graph $G$, i.e., $P_{\mathcal{S}(G)} = \text{conv}(\left\{\sum_{s \in S} \bm{e}_{s}: S \in \mathcal{S}(G)\right\} )$ (see \cite{PerfElim}). 
Note that because chordal graphs are perfect, there is a complete inequality description of $P_{\mathcal{S}}$ using only clique inequalities.
It is also well-known that a graph is chordal if and only if it has a \emph{perfect elimination ordering} of its vertices \cite{PerfElim}, namely an ordering of the 
vertices of the graph such that, for each vertex $v$, $v$ and the neighbors of $v$ that occur after $v$ in the order form a clique. This can be 
interpreted as a sequence of cliques of increasing sizes. Thus, one can use the perfect elimination orderings to obtain a maximum-size clique of a 
chordal graph in polynomial-time. Furthermore, it was shown in \cite{PerfElimEff} that such an ordering may be found efficiently. One can check 
that the sequence of vertices obtained by the perfect elimination ordering coincides with the steps taken by the Ordered Shadow 
rule so long as the perfect elimination ordering coincides with the ordering of the indices in the corresponding 0/1 polytope.

In \cite{Kortenkampetal97}, the authors showed that there exist two-dimensional projections of $0/1$ polytopes with exponentially many vertices. Hence, the original Shadow pivot rule may take an exponential number of iterations. Since the Slim Shadow rule requires a number of steps bounded by $n$, the dimension of the ambient space, this suggests the question: Is there an example in which the length of the path chosen by the Slim Shadow rule is exponential in the dimension $d$ of the feasible region? Unfortunately, yes. The example in \cite{Kortenkampetal97} can be modified to yield an explicit set of LPs for which the Slim Shadow rule requires a number of steps exponential in $d$ (while of course still being bounded by the number of variables).

There is no canonical, universal winner on performance between our two Shadow rules. At least when applied as we described them here,
in some cases the Slim Shadow rule may actually perform better than the Ordered Shadow rule. In particular, the bounds from sparsity in Lemma \ref{lem:sparsity} may be stronger than the dimension bound. To see this note that the Birkhoff polytope for $n \times n$ permutation matrices has dimension $(n-1)^2$. This is the bound on the number of steps for the Ordered Shadow rule, yet the Slim Shadow rule achieves a bound of only $n$ steps by Lemma \ref{lem:sparsity}. 

One can ask, how good are the bounds we obtain on the length of monotone paths compared to the optimal bounds? Let us compare the Slim Shadow rule in a few instances: 

\begin{itemize}
    \item The rule yields at most $n$ steps on both the asymmetric and symmetric traveling salesman polytopes for $n$ vertex graphs. The respective optimal bounds are $\lfloor n/2 \rfloor$ and $\lfloor n/3 \rfloor$ (see \cite{rispoli1998}).
    \item The rule yields at most $n$ steps on the Birkhoff polytope for $n \times n$ matrices. The optimal bound is $\lfloor n/2 \rfloor$  (\cite{rispoli1992}).
    \item The rule yields at most $\lfloor n/2 \rfloor$ steps for the perfect matching polytope on the complete graph with $n$ vertices. The
    optimal bound is $\lfloor n/4 \rfloor$ (\cite{rispoli1992}).
    \item The rule yields at most $\text{rank}(\mathcal{M})$ for an independent set matroid polytope of a matroid $\mathcal{M}$ (starting at $\bm{0}$). As we saw this matches the optimal bound.
\end{itemize}

Thus the lengths we obtained are \textendash\, up to a constant \textendash\, the same as the actual monotone diameter in 
all of the above cases. Note that the bounds we find do not require any knowledge of the combinatorics of the graphs of any of these polytopes, 
yet they remain not far off from the best possible bounds.

To conclude, we stress that the analysis we did in this paper is only about bounding the number of \textit{edge steps} and \emph{non-degenerate pivots}. 

\begin{question} 
Is there a polynomial bound to the number of  \emph{degenerate} pivots performed by any of the pivot rules we presented when applied to $0/1$-LPs?
\end{question}

\paragraph{Acknowledgments} The third author is grateful for the support received from the NWO-VIDI grant VI.Vidi.193.087.
The first and second authors are grateful for the support received through NSF grants DMS-1818969 and the NSF GRFP.

\bibliographystyle{plain}
\bibliography{01simplex.bib}
\end{document}